\documentclass[11pt, reqno]{amsart}

\usepackage{amsmath, amsfonts, amssymb, amsbsy, bigstrut, graphicx, enumerate,  upref, longtable, comment, booktabs, array, caption, subcaption}

\usepackage{pstricks}

\allowdisplaybreaks

\usepackage[breaklinks]{hyperref}

\usepackage[numbers, sort&compress]{natbib}

\newtheorem{thm}{Theorem}[section]
\newtheorem{lmm}[thm]{Lemma}
\newtheorem{cor}[thm]{Corollary}
\newtheorem{prop}[thm]{Proposition}

\theoremstyle{definition}

\newcommand{\cov}{\mathrm{Cov}}

\newcommand{\ee}{\mathbb{E}}

\newcommand{\mf}{\mathcal{F}}

\newcommand{\pp}{\mathbb{P}}

\newcommand{\rr}{\mathbb{R}}

\newcommand{\var}{\mathrm{Var}}
\newcommand{\ve}{\varepsilon}

\numberwithin{equation}{section}

\usepackage[utf8]{inputenc}
\usepackage{amsmath}
\usepackage{amsfonts}
\usepackage{bbm}
\usepackage{mathtools}
\usepackage{tikz}
\usepackage{amsthm}
\usepackage[english]{babel}
\usepackage{fancyhdr}
\usepackage{amssymb}
\usepackage{enumitem}
\usepackage{qtree}
\usepackage{mathrsfs}

\renewcommand{\tilde}{\widetilde}
\renewcommand{\hat}{\widehat}

\begin{document}

%\title[Ultrametrcity and hierarchical organization]{From ultrametricity to hierarchical organization of states}
\title{A new coefficient of correlation}
\author{Sourav Chatterjee}
\address{Department of Statistics, Stanford University, Sequoia Hall, 390 Jane Stanford Way, Stanford, CA 94305}
\email{souravc@stanford.edu}
\thanks{Research partially supported by NSF grant DMS-1855484}

\begin{abstract}
Is it possible to define a coefficient of correlation which is (a) as simple as the classical coefficients like Pearson's correlation or Spearman's correlation, and yet (b) consistently estimates some simple and interpretable measure of the degree of dependence between the variables,  which is~$0$ if and only if the variables are independent and~$1$ if and only if one is a measurable function of the other, and (c) has a simple asymptotic theory under the hypothesis of independence, like the classical coefficients? This article answers this question in the affirmative, by producing such a coefficient. No assumptions are needed on the distributions of the variables. There are several coefficients in the literature that converge to $0$ if and only if the variables are independent, but none that satisfy any of the other properties mentioned above. 
\end{abstract}

\maketitle

%\tableofcontents

\section{Introduction}
The three most popular classical measures of statistical association are Pearson's correlation coefficient, Spearman's $\rho$, and Kendall's $\tau$. These coefficients are very powerful for detecting linear or monotone associations, and they have well-developed asymptotic theories for calculating P-values. However, the big problem is that they are not effective for detecting associations that are not monotonic, even in the complete absence of noise. 

There have been many proposals to address this deficiency of the classical coefficients~\cite{Josse16}, such as the maximal correlation coefficient~\cite{hirschfeld35, gebelein41, Renyi59, Breiman85}, various coefficients based on joint cumulative distribution functions and ranks~\cite{gkl18, dhs18, hcl17, bd14, nwd16, wdl16, wdm18, yanagimoto70, csorgo85, purisen71, hoeffding48, bkr61, romano88, rosenblatt75, debsen19, wjl17}, kernel-based methods~\cite{pbsp18, gretton05, gretton08, sen14, zfgs18}, information theoretic coefficients~\cite{kraskov04, Linfoot57, Reshef11}, coefficients based on copulas~\cite{dss13, lopezpaz13, sklar59, Schweizer81, zhang19}, and coefficients based on  pairwise distances~\cite{Szekely07, sr09, hhg13,fr83,lyons13}.  

Some of these coefficients are popular among practitioners. But there are two common problems. First, most of these coefficients are designed for testing independence, and not for measuring the strength of the relationship between the variables. Ideally, one would like a coefficient that approaches its maximum value if and only if one variable looks more and more like a noiseless  function of the other, just as  Pearson correlation is close to its maximum value if and only if one variable is close to being a noiseless {\it linear} function of the other. It is sometimes believed that the maximal information coefficient~\cite{Reshef11} and the maximal correlation coefficient~\cite{Renyi59} measure the strength of the relationship in the above sense, but we will see later in Section \ref{micsec} that that's not necessarily correct. Although they are maximized when one variable is a function of the other, the converse is not true. They may be equal to $1$ even if the relationship is very noisy.

Second, most of these coefficients do not have simple asymptotic theories under the hypothesis of independence that facilitate the quick computation of P-values for testing independence. In the absence of such theories, the only recourse is to use computationally expensive permutation tests or other kinds of bootstrap. %Even when they do have some amount of asymptotic theory, it is usually too complicated to be useful in practice and is not implemented in software. 

In this situation, one may wonder if it is at all possible to define a coefficient that is (a) as simple as the classical coefficients, and yet (b) is a consistent estimator of some measure of dependence which is $0$ if and only if the variables are independent and  $1$ if and only if one is a measurable function of the other, and (c) has a simple asymptotic theory under the hypothesis of independence, like the classical coefficients. 

Such a coefficient is presented below. The formula is so simple that it is likely that there are many such coefficients, some of them possibly having better properties than the one presented below.

Let $(X,Y)$ be a pair of random variables, where $Y$ is not  a constant.  Let $(X_1,Y_1),\ldots,(X_n,Y_n)$ be i.i.d.~pairs with the same law as $(X,Y)$, where $n\ge 2$. The new coefficient has a simpler formula if the $X_i$'s and the $Y_i$'s have no ties. This simpler formula is presented first, and then the general case is given. Suppose that the $X_i$'s and the $Y_i$'s have no ties. Rearrange the data as $(X_{(1)},Y_{(1)}),\ldots,(X_{(n)}, Y_{(n)})$ such that $X_{(1)}\le \cdots \le X_{(n)}$. Since the $X_i$'s have no ties, there is a unique way of doing this. Let $r_i$ be the rank of $Y_{(i)}$, that is, the number of $j$ such that $Y_{(j)}\le Y_{(i)}$. The new correlation coefficient is defined as
\begin{equation}\label{xindef1}
\xi_n(X,Y) := 1-\frac{3\sum_{i=1}^{n-1}|r_{i+1}-r_i|}{n^2-1}.
\end{equation}
In the presence of ties, $\xi_n$ is defined as follows. If there are ties among the $X_i$'s, then choose an increasing rearrangement as above by breaking ties uniformly at random. Let $r_i$ be as before, and additionally define $l_i$ to be the number of $j$ such that $Y_{(j)}\ge Y_{(i)}$. Then define
\[
\xi_n(X,Y) := 1-\frac{n\sum_{i=1}^{n-1}|r_{i+1}-r_i|}{2\sum_{i=1}^n l_i(n-l_i)}. 
\]
When there are no ties among the $Y_i$'s, $l_1,\ldots,l_n$ is just a permutation of $1,\ldots,n$, and so the denominator in the above expression is just $n(n^2-1)/3$, which reduces this definition to the earlier expression \eqref{xindef1}. 

The following theorem shows that $\xi_n$ is a consistent estimator of a certain measure of dependence between the random variables $X$ and $Y$. 

\begin{thm}\label{mainthm}
If $Y$ is not almost surely a constant, then as $n\to\infty$, $\xi_n(X,Y)$ converges almost surely to the deterministic limit
\begin{equation}\label{xidef}
\xi(X,Y) := \frac{\int \var(\ee(1_{\{Y\ge t\}}|X)) d\mu(t)}{\int\var(1_{\{Y\ge t\}}) d\mu(t)},
\end{equation}
where $\mu$ is the law of $Y$. This limit belongs to the interval $[0,1]$. It is  $0$ if and only if $X$ and $Y$ are independent, and it is $1$ if and only if there is a measurable function $f:\rr\to\rr$ such that $Y=f(X)$ almost surely. 
\end{thm}
%There are probably many questions in the reader's mind at this point. I will try to answer some of them in the following remarks.
%\begin{enumerate}
%\item 

{\it Remarks.} (1) Unlike most coefficients, $\xi_n$ is not symmetric in $X$ and $Y$. But that is intentional. We would like to keep it that way  because we may want to understand if $Y$ is a function $X$, and not just if one of the variables is a function of the other. If we want to understand whether $X$ is a function of $Y$, we should use $\xi_n(Y,X)$ instead of $\xi_n(X,Y)$. A symmetric measure of dependence, if required, can be easily obtained by taking the maximum of $\xi_n(X,Y)$ and $\xi_n(Y,X)$. By Theorem \ref{mainthm}, this symmetrized coefficient converges in probability to $\max\{\xi(X,Y), \xi(Y,X)\}$, which is $0$ if and only if $X$ and $Y$ are independent, and $1$ if and only if at least one of $X$ and $Y$ is a measurable function of the other.

%\item 
(2) It is clear that $\xi(X,Y)\in [0,1]$ since $\var(1_{\{Y\ge t\}})\ge \var(\ee(1_{\{Y\ge t\}}|X))$ for every $t$. If $X$ and $Y$ are independent, then $\ee(1_{\{Y\ge t\}}|X)$ is a constant, and therefore $\xi(X,Y)=0$. If $Y$ is a measurable function of $X$, then  $\ee(1_{\{Y\ge t\}}|X) = 1_{\{Y\ge t\}}$, and so $\xi(X,Y)=1$. The converse implications will be proved in Section \ref{proof1}. The most non-obvious part of Theorem \ref{mainthm} is the convergence of $\xi_n(X,Y)$ to $\xi(X,Y)$. The proof of this, given in Section \ref{proof1}, is quite lengthy. For the convenience of the reader (and to facilitate possible future improvements), a brief sketch of the proof is given in Section~\ref{sketchsec}.
%\item 

(3) In Theorem \ref{mainthm}, there are no restrictions on the law of $(X,Y)$ other than that $Y$ is not a constant. In particular, $X$ and $Y$ can be discrete, continuous, light-tailed or heavy-tailed.

%\item 

(4) The coefficient $\xi_n(X,Y)$ remains unchanged if we apply strictly increasing transformations  to $X$ and $Y$, because it is based on ranks.  For the same reason, it can be computed in time $O(n\log n)$. We will see later that the actual computation on a computer is also very fast. The cost that we have to pay for fast computability, as we will see in Section \ref{powersec}, is that the test of independence based on $\xi_n$ is sometimes less powerful than tests based on statistics whose computational times are quadratic in the sample size.

(5) The limiting value $\xi(X,Y)$ has appeared earlier in the literature~\cite{dss13, gkl18}. The paper \cite{dss13} gives a copula-based estimator for $\xi(X,Y)$ when $X$ and $Y$ are continuous, that is consistent under smoothness assumptions on the copula and appears to be computable in time $n^{5/3}$ for an optimal choice of tuning parameters. %Other than this, I am not aware of any other coefficient that converges to $0$ if and only if $X$ and $Y$ are independent and to $1$ if and only if $Y$ is a measurable function of $X$ (although there are numerous coefficients that satisfy only the first criterion). 

%\item 
%(6) There are many coefficients in the literature that provably converge to $0$ if and only if $X$ and $Y$ are independent, although none as simple as $\xi_n$. However, I am not aware of any other coefficient that converges to  a simple and interpretable measure of the degree of dependence between $X$ and $Y$ such as $\xi(X,Y)$. This is a generalization of the property of Pearson's correlation that it converges to the population correlation $\textup{Cor}(X,Y)$, which is a simple and interpretable measure of the degree of {\it linear} relationship between $X$ and $Y$. 
%(6) It is sometimes mistakenly believed that the maximal information coefficient (MIC)~\cite{Reshef11} and the maximal correlation coefficient~\cite{Renyi59} converge to $1$ if and only if one variable is a function of the other. But that is not true, as we will see in Section \ref{micsec}. %Similarly, another recently proposed coefficient called `generalized R-squared'~\cite{wjl17} converges to $0$ if and only if conditional expectations and variances are constant, which is not the same as independence. %It seems that $\xi_n$ may be the only coefficient proposed until now that truly measures the {\it strength of the relationship} between $X$ and $Y$ in the most general sense and under the least possible set of assumptions. 

%\item 
(6) The coefficient $\xi_n$ looks similar to some coefficients defined earlier~\cite{fr83, sg18}, but in spite of its simple form, it seems to be genuinely new. %, but in spite of its simple form, it seems to be a genuinely new measure of association. For example, the coefficient $\Gamma_2$ defined in~\cite{fr83} looks similar, but it can be shown through counterexamples that $\xi_n$ and $\Gamma_2$ are not related (that is, one is not a function of the other). The same is true for coefficients defined recently in~\cite{sg18}.  

%\item 
(7) Multivariate measures of dependence and conditional dependence inspired by $\xi_n$ are now  available in the preprint~\cite{ac19}. 

%\item 
%(8) If $(X,Y)$ is bivariate normal with correlation $\rho$, it is not hard to show that $\xi(X,Y)$ is like a constant times $\rho^2$ when $\rho$ is small. %Therefore for large $n$, $\xi_n(X,Y)$ should also behave similarly.

%\item 
(8) If the $X_i$'s have ties, then $\xi_n(X,Y)$ is a randomized estimate of $\xi(X,Y)$, because of the randomness coming from the breaking of ties. This can be ignored if $n$ is large, because $\xi_n$ is guaranteed to be close to $\xi$ by Theorem \ref{mainthm}. Alternatively, one can consider taking the average of $\xi_n$ over all possible increasing rearrangements of the $X_i$'s. %, which can be probably be computed in time $n^2$.%which has an expression that is not too complicated. However, the computational time for this expression is of order $n^2$, which is why I have avoided that route. 

%\item 
(9) If there are no ties among the $Y_i$'s, the maximum possible value of $\xi_n(X,Y)$ is $(n-2)/(n+1)$, which is attained if $Y_i=X_i$ for all $i$. This can be noticeably less than $1$ for small $n$. For example, for $n=20$, this value is approximately $0.86$. Users should be aware of this fact about $\xi_n$. On the other hand, it is not very hard to prove that the minimum possible value of $\xi_n(X,Y)$ is $-1/2+O(1/n)$, and the minimum is attained when the top $n/2$ values of $Y_i$ are placed alternately with the bottom $n/2$ values. This seems to be paradoxical, since Theorem \ref{mainthm} says that the limiting value is in $[0,1]$. The resolution is that Theorem \ref{mainthm} only applies to i.i.d.~samples. Therefore a large  negative value of $\xi_n$ has only one possible interpretation: the data does not resemble an i.i.d.~sample.

(10) An R package for calculating $\xi_n$ and P-values for testing independence (based on the theory presented in the next section), named XICOR, is available on CRAN~\cite{ch20}. %Temporarily, the code is available at  \url{https://statweb.stanford.edu/~souravc/xi.R}.

%(12) An R package for calculating $\xi_n$ and P-values for testing independence (based on the theory presented in the next section) is in preparation, in collaboration with Susan Holmes. Temporarily, the code is available at  \url{https://statweb.stanford.edu/~souravc/xi.R}.

%\end{enumerate}
%\noindent {\it Remarks.} (1) To see that $\xi_n$ is as simple as Spearman's $\rho$, notice that in our notation, 
%\[
%\text{Spearman's $\rho$ } = 1- \frac{6\sum_{i=1}^n (r_i-i)^2}{n(n^2-1)}. 
%\]
%Spearman's $\rho$  is close to $1$ when $r_i$ is close to $i$ for most $i$; in other words, when $Y$ is approximately like an increasing function of $X$.  On the other hand, $\xi_n$ is close to $1$ when $r_i$ is close to $r_{i+1}$ for most $i$; that is, when  $Y$ is like a continuous function of $X$ at most places. This does not explain, however, why $\xi_n\approx 0$ if and only if $X$ and $Y$ are approximately independent. I am not sure how to explain that without explaining the details of the proof. 

%\vskip.2in

\section{Testing independence}
The main purpose of $\xi_n$ is to provide a measure of the strength of the relationship between $X$ and $Y$, and not to serve as a test statistic for testing independence. However, one can use it for testing independence if so desired. In fact, it has a nice and simple asymptotic theory under independence. The next theorem gives the asymptotic distribution of $\sqrt{n}\xi_n$ under the hypothesis of independence and the assumption that $Y$ is continuous. The more general asymptotic theory in the absence of continuity is presented after that. 

\begin{thm}\label{cltthm0}
Suppose that $X$ and $Y$ are independent and $Y$ is continuous. Then  $\sqrt{n}\xi_n(X,Y)\to N(0,2/5)$ in distribution as $n\to \infty$.
\end{thm}
The above result is essentially a restatement the main theorem of~\cite{cbl93}, where a similar statistic for measuring the `presortedness' of a permutation was studied. We will see later in numerical examples that the convergence in Theorem~\ref{cltthm0} happens quite fast. It is roughly valid even for $n$ as small as $20$. 

%In the rest of this section, I will describe the asymptotic theory when $Y$ is not continuous, which is just a little bit more complicated. The reader may skip this part at first reading and move on to the next section.

If $X$ and $Y$ are independent but $Y$ is not continuous, then also $\sqrt{n}\xi_n$ converges in distribution to a centered Gaussian law, but the variance has a more complicated expression, and may depend on the law of $Y$. For each $t\in \rr$, let $F(t):= \pp(Y\le t)$ and $G(t):=\pp(Y\ge t)$. Let $\phi(y,y') := \min\{F(y), F(y')\}$. Define 
\begin{equation}\label{tauform}
\tau^2 = \frac{\ee \phi(Y_1,Y_2)^2 - 2\ee(\phi(Y_1,Y_2)\phi(Y_1,Y_3)) + (\ee\phi(Y_1,Y_2))^2}{(\ee G(Y)(1-G(Y)))^2},
\end{equation}
where $Y_1,Y_2,Y_3$ are independent copies of $Y$. The following theorem  generalizes Theorem \ref{cltthm0}.
\begin{thm}\label{cltthm}
Suppose that $X$ and $Y$ are independent. Then   $\sqrt{n}\xi_n(X,Y)$ converges to $N(0, \tau^2)$ in distribution as $n\to\infty$, where $\tau^2$ is given by the formula~\eqref{tauform} stated above. The number $\tau^2$ is strictly positive if $Y$ is not a constant, and equals $2/5$ if $Y$ is continuous. 
\end{thm}

The simple reason why $\tau^2$ does not depend on the law of $Y$ if $Y$ is continuous is that in this case $F(Y)$ and $G(Y)$ are Uniform$[0,1]$ random variables, which implies that the expectations in \eqref{tauform} do not depend on the law of $Y$. If $Y$ is not continuous, then $\tau^2$ may depend on the law of $Y$. For example, it is not hard to show that if $Y$ is a Bernoulli$(1/2)$ random variable, then $\tau^2=1$. Fortunately, if $Y$ is not continuous, there is a simple way to estimate $\tau^2$ from the data using the estimator
\[
\hat{\tau}^2_n = \frac{a_n-2b_n+c_n^2}{d_n^2},
\]
where  $a_n$, $b_n$, $c_n$ and $d_n$ are defined as follows. For each $i$, let 
\begin{equation}\label{rldef}
R(i) := \#\{j: Y_j\le Y_i\}, \ \ \ L(i) := \#\{j: Y_j\ge Y_i\}. 
\end{equation}
Let $u_1\le u_2\le\cdots \le u_n$ be an increasing rearrangement of $R(1),\ldots, R(n)$. Let $v_i := \sum_{j=1}^i u_j$ for $i=1,\ldots,n$.  Define
\begin{align*}
&a_n := \frac{1}{n^4}\sum_{i=1}^n (2n-2i+1) u_i^2, \ \ \ b_n := \frac{1}{n^5}\sum_{i=1}^n (v_i + (n-i)u_i)^2, \\
&c_n := \frac{1}{n^3}\sum_{i=1}^n (2n-2i+1)u_i, \ \ \  d_n := \frac{1}{n^3}\sum_{i=1}^n L(i)(n-L(i)). 
\end{align*}
Then we have the following result.
%The following theorem shows that $\hat{\tau}_n^2$ is a consistent estimator of $\tau^2$, and it is computable quickly. 
\begin{thm}\label{estthm}
The estimator $\hat{\tau}_n^2$ can be computed in time $O(n\log n)$, and converges to $\tau^2$ almost surely as $n\to\infty$. 
\end{thm}
I do not have the asymptotic theory for $\xi_n(X,Y)$ when $X$ and $Y$ are dependent.  Simulation results  presented in Section~\ref{asympsec} indicate that even under dependence, $\sqrt{n}(\xi_n-\xi)$ is asymptotically normal. 
%This appears to be quite a challenging problem.  Such a theory is needed if for some reason one wants to build a confidence interval for $\xi(X,Y)$. Preliminary numerical investigations indicate that this cannot be done by bootstrapping, because $\xi_n$ behaves quite differently in a bootstrapped sample than in the original sample. The probable reason is that $\xi_n$ is able to detect extra dependence between the coordinates in the bootstrapped sample. 

One may also ask about the asymptotic null distribution of the symmetrized statistic $\max\{\xi_n(X,Y), \xi_n(Y,X)\}$. It is likely that under independence, this behaves like the maximum of a pair of correlated normal random variables. At this time I do not have a proof of this claim, nor a conjecture about the parameters of this distribution. Of course, it is easy to carry out a permutation test for independence using the symmetrized statistic. %But it is also likely that the proof of such a claim will be long and challenging. Since this paper is already quite lengthy, I will not attempt to do it here. 

The rest of the paper is organized as follows. We begin with an amusing  application of $\xi_n$ to Galton's peas data in Section \ref{galtonsec}. Various simulation results are presented in Section \ref{numsec}. An application to a famous gene expression dataset is given in Section \ref{spellmansec}. The inadequacy of  MIC and maximal correlation for measuring the strength of relationship between $X$ and $Y$ is proved in Section \ref{micsec}. A summary of the advantages and disadvantages of using $\xi_n$ is given in Section \ref{adsec}. A sketch of the proof of Theorem \ref{mainthm} is given in Section~\ref{sketchsec}. The remaining sections are devoted to proofs.

\section{Example: Galton's peas revisited}\label{galtonsec}

\begin{figure}[t]
\includegraphics[width = .5\textwidth]{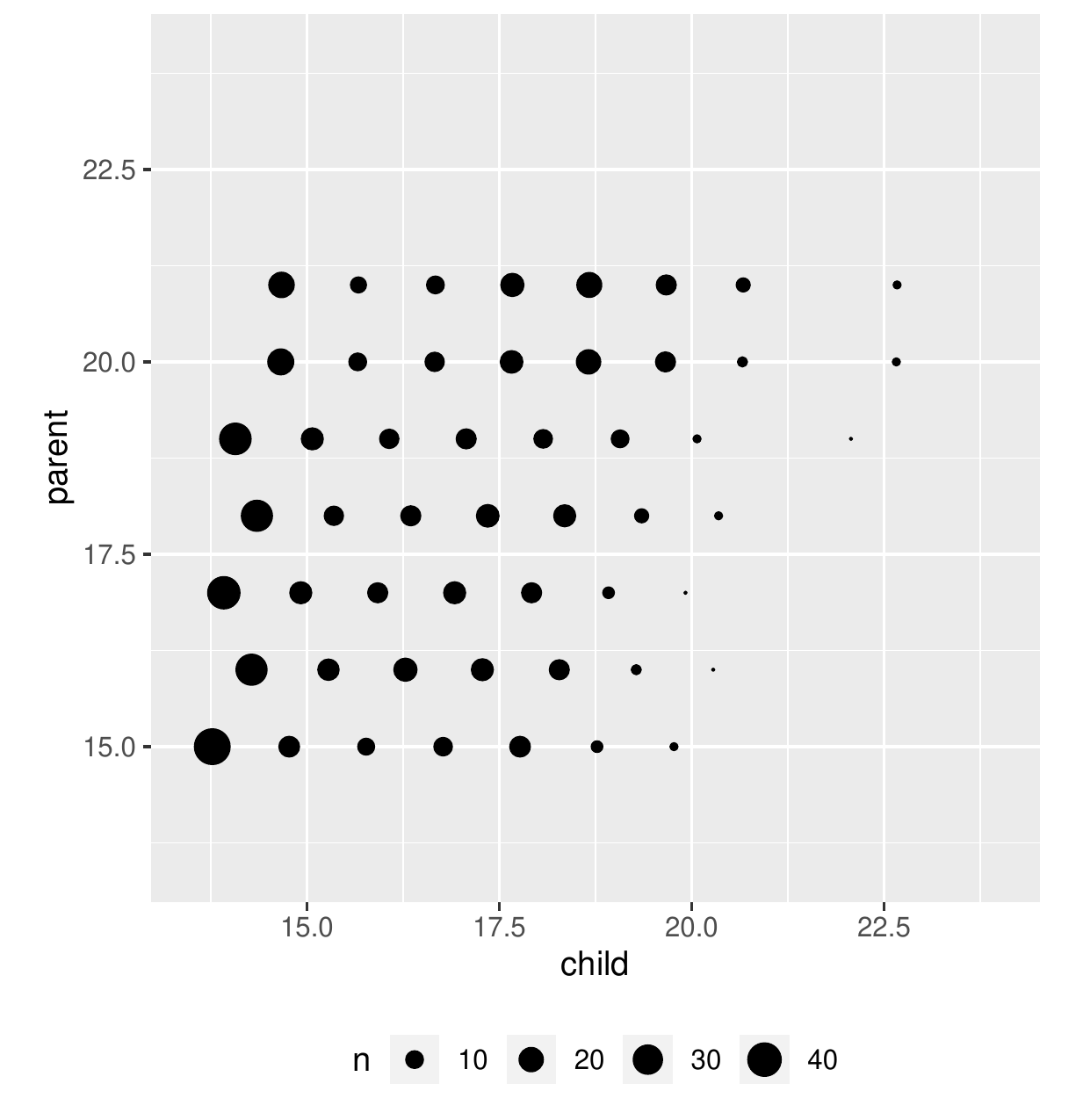}
\caption{Scatterplot of Galton's peas data. Thickness of a dot represents the number of data points at that location. (Figure courtesy of Susan Holmes.)\label{figpeas}}
\end{figure}

Sir Francis Galton's peas data, collected in 1875, is one of the earliest and most famous datasets in the history of statistics. The data consists of $700$ observations of mean diameters of sweet peas in mother plants and daughter plants.  The exact process of data collection was not properly recorded; all we know is that Galton sent out packets of seeds to friends, who planted the seeds, grew the plants, and sent the seeds from the new plants back to Galton (see \cite[p.~296]{stigler86} for further details). The dataset is freely available as the `peas' data frame in the psych package in R.

Let $X$ be the mean diameter of peas in a mother plant, and $Y$ be the mean diameter of peas in the daughter plant. As already observed by Pearson long ago, the correlation between $X$ and $Y$ is around $0.35$. The $X_i$'s have many ties in this data, which means that $\xi_n(X,Y)$ is random due to the random breaking of ties. Averaging over ten thousand simulations gave a value close to $0.11$ for $\xi_n(X,Y)$. The P-value for the test of independence using Theorems \ref{cltthm} and \ref{estthm} came out to be less than $0.0001$, so $\xi_n(X,Y)$ succeeded in the task of detecting dependence between $X$ and $Y$.

\begin{table}[t]
\centering
\begin{tiny}
\caption{Contingency table for Galton's peas data.\label{galton}}
\begin{tabular}{rrrrrrrr}
  \toprule
  & \multicolumn{7}{c}{Parent}\\
  \cmidrule{2-8}
 Child & 15 & 16 & 17 & 18 & 19 & 20 & 21 \\ 
  \midrule
13.77 &  46 &   0 &   0 &   0 &   0 &   0 &   0 \\ 
  13.92 &   0 &   0 &  37 &   0 &   0 &   0 &   0 \\ 
  14.07 &   0 &   0 &   0 &   0 &  35 &   0 &   0 \\ 
  14.28 &   0 &  34 &   0 &   0 &   0 &   0 &   0 \\ 
  14.35 &   0 &   0 &   0 &  34 &   0 &   0 &   0 \\ 
  14.66 &   0 &   0 &   0 &   0 &   0 &  23 &   0 \\ 
  14.67 &   0 &   0 &   0 &   0 &   0 &   0 &  22 \\ 
  14.77 &  14 &   0 &   0 &   0 &   0 &   0 &   0 \\ 
  14.92 &   0 &   0 &  16 &   0 &   0 &   0 &   0 \\ 
  15.07 &   0 &   0 &   0 &   0 &  16 &   0 &   0 \\ 
  15.28 &   0 &  15 &   0 &   0 &   0 &   0 &   0 \\ 
  15.35 &   0 &   0 &   0 &  12 &   0 &   0 &   0 \\ 
  15.66 &   0 &   0 &   0 &   0 &   0 &  10 &   0 \\ 
  15.67 &   0 &   0 &   0 &   0 &   0 &   0 &   8 \\ 
  15.77 &   9 &   0 &   0 &   0 &   0 &   0 &   0 \\ 
  15.92 &   0 &   0 &  13 &   0 &   0 &   0 &   0 \\ 
  16.07 &   0 &   0 &   0 &   0 &  12 &   0 &   0 \\ 
  16.28 &   0 &  18 &   0 &   0 &   0 &   0 &   0 \\ 
  16.35 &   0 &   0 &   0 &  13 &   0 &   0 &   0 \\ 
  16.66 &   0 &   0 &   0 &   0 &   0 &  12 &   0 \\ 
  16.67 &   0 &   0 &   0 &   0 &   0 &   0 &  10 \\ 
  16.77 &  11 &   0 &   0 &   0 &   0 &   0 &   0 \\ 
  16.92 &   0 &   0 &  16 &   0 &   0 &   0 &   0 \\ 
  17.07 &   0 &   0 &   0 &   0 &  13 &   0 &   0 \\ 
  17.28 &   0 &  16 &   0 &   0 &   0 &   0 &   0 \\ 
  17.35 &   0 &   0 &   0 &  17 &   0 &   0 &   0 \\ 
   \bottomrule
\end{tabular}
\hskip.3in
\begin{tabular}{rrrrrrrr}
  \toprule
  & \multicolumn{7}{c}{Parent}\\
  \cmidrule{2-8}
 Child & 15 & 16 & 17 & 18 & 19 & 20 & 21 \\ 
  \midrule
  17.66 &   0 &   0 &   0 &   0 &   0 &  17 &   0 \\ 
  17.67 &   0 &   0 &   0 &   0 &   0 &   0 &  18 \\ 
  17.77 &  14 &   0 &   0 &   0 &   0 &   0 &   0 \\ 
  17.92 &   0 &   0 &  13 &   0 &   0 &   0 &   0 \\ 
  18.07 &   0 &   0 &   0 &   0 &  11 &   0 &   0 \\ 
  18.28 &   0 &  13 &   0 &   0 &   0 &   0 &   0 \\ 
  18.35 &   0 &   0 &   0 &  16 &   0 &   0 &   0 \\ 
  18.66 &   0 &   0 &   0 &   0 &   0 &  20 &   0 \\ 
  18.67 &   0 &   0 &   0 &   0 &   0 &   0 &  21 \\ 
  18.77 &   4 &   0 &   0 &   0 &   0 &   0 &   0 \\ 
  18.92 &   0 &   0 &   4 &   0 &   0 &   0 &   0 \\ 
  19.07 &   0 &   0 &   0 &   0 &  10 &   0 &   0 \\ 
  19.28 &   0 &   3 &   0 &   0 &   0 &   0 &   0 \\ 
  19.35 &   0 &   0 &   0 &   6 &   0 &   0 &   0 \\ 
  19.66 &   0 &   0 &   0 &   0 &   0 &  13 &   0 \\ 
  19.67 &   0 &   0 &   0 &   0 &   0 &   0 &  13 \\ 
  19.77 &   2 &   0 &   0 &   0 &   0 &   0 &   0 \\ 
  19.92 &   0 &   0 &   1 &   0 &   0 &   0 &   0 \\ 
  20.07 &   0 &   0 &   0 &   0 &   2 &   0 &   0 \\ 
  20.28 &   0 &   1 &   0 &   0 &   0 &   0 &   0 \\ 
  20.35 &   0 &   0 &   0 &   2 &   0 &   0 &   0 \\ 
  20.66 &   0 &   0 &   0 &   0 &   0 &   3 &   0 \\ 
  20.67 &   0 &   0 &   0 &   0 &   0 &   0 &   6 \\ 
  22.07 &   0 &   0 &   0 &   0 &   1 &   0 &   0 \\ 
  22.66 &   0 &   0 &   0 &   0 &   0 &   2 &   0 \\ 
  22.67 &   0 &   0 &   0 &   0 &   0 &   0 &   2 \\ 
   \bottomrule
\end{tabular}
\end{tiny}
\end{table}

Thus far, there is nothing surprising.  The real surprise, however, was that the value of $\xi_n(Y,X)$ (instead of $\xi_n(X,Y)$) turned out to be approximately $0.92$ (and it appeared to be independent of the tie-breaking process). By Theorem~\ref{mainthm}, this means that $X$ is close to being  a noiseless function of $Y$. From the scatterplot of the data (Figure~\ref{figpeas}), it is not clear how this can be possible.  The mystery is resolved by looking at the contingency table of the data (Table~\ref{galton}). Each row of the table corresponds to a value of $Y$, and each column corresponds to a value of $X$. We notice that each column has multiple cells with nonzero counts, meaning that for each value of $X$ there are many different values of $Y$ in the data. On the other hand, each row in the table contains exactly one cell with a nonzero (and often quite large) count. That is, for any value of $Y$, every value of $X$ in the data is the same.

For example, among all mother plants with mean diameter $15$, there were $46$ cases where the daughter plant had diameter $13.77$, $14$ had diameter $14.77$, $11$ had diameter $16.77$, $14$ had diameter $17.77$, and $4$ had diameter $18.77$. On the other hand, for all $46$ daughter plants in the data  with  diameter $13.77$, the mother plants had diameter $15$. Similarly, for all $34$ daughter plants with diameter $14.28$, the mother plants had diameter $16$. 

Common sense suggests that the reason behind this strange phenomenon is surely some quirk of the data collection or recording method, and not some profound biological fact. (It is probably not a simple rounding effect, though; for instance, in all $46$ cases where $Y = 13.77$, we have $X=15$, but for all $37$ cases where $Y=13.92$, which is only slightly different than $13.77$, we have $X=17$.) However, if we imagine that the values recorded in the data are the exact values that were measured and the observations were i.i.d.~(neither of which is exactly true, as I learned from Steve Stigler), then looking at Table \ref{galton} there is no way to escape the conclusion that the mean diameter of peas in the mother plant can be exactly predicted with considerable certainty by the mean diameter of the peas in the daughter plant (but not the other way around). The coefficient $\xi_n(Y,X)$ discovers this fact numerically by attaining a value close to $1$. It is probable that this feature of Galton's peas data has been noted before, but if so, it is certainly hard to find. I could not find any reference where this is mentioned, in spite of much effort.

\section{Simulation results}\label{numsec}
The goal of this section is to investigate the performance of $\xi_n$ using  numerical simulations, and compare it to other methods. We compare general performance, run times, and powers for testing independence. 

\subsection{General performance, equitability and generality}

\begin{figure}[t]
\centering
\begin{subfigure}[b]{.3\textwidth}
\centering
\includegraphics[width = \textwidth]{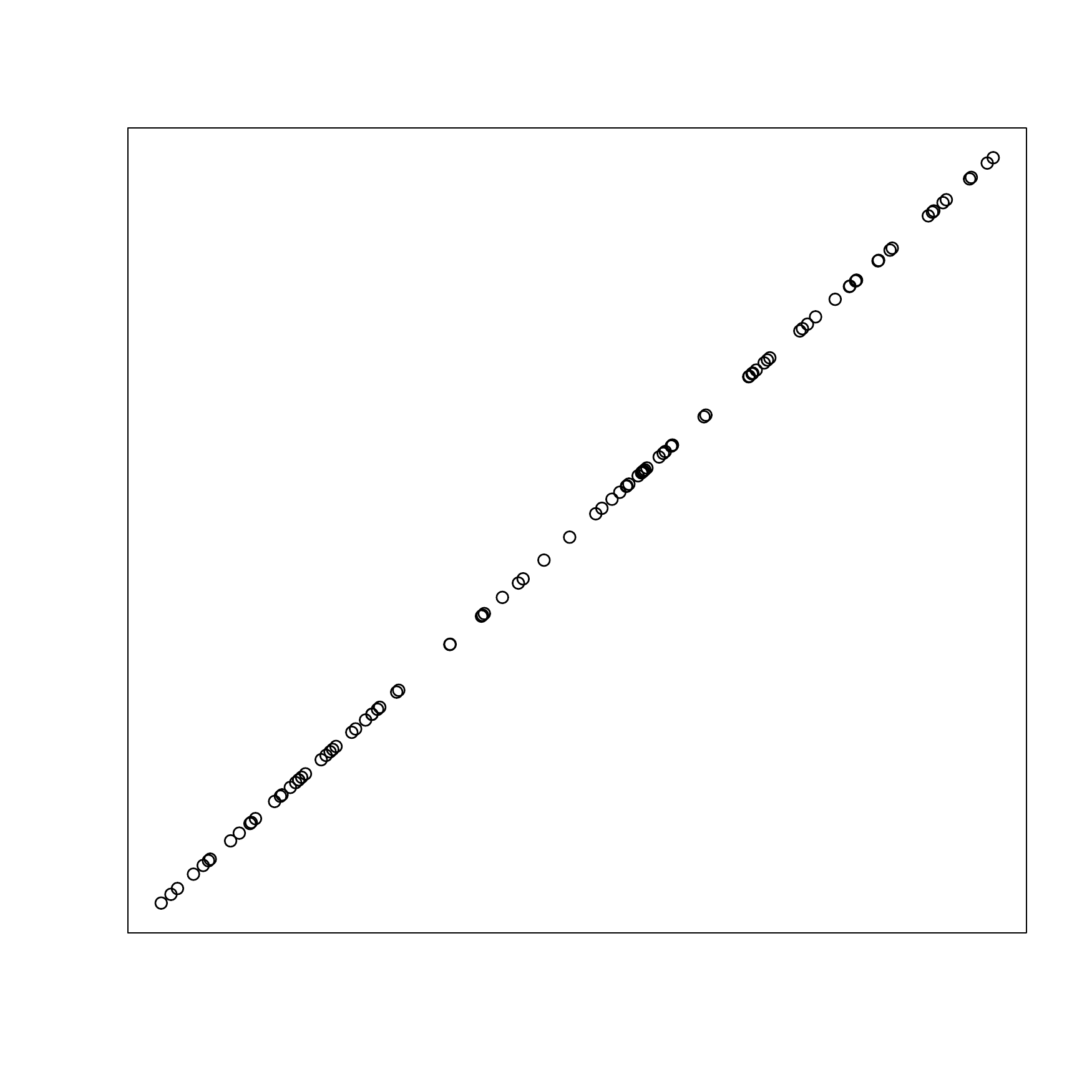}
\caption{$\xi_n = 0.970$. \label{fig0a}}
\end{subfigure}
\begin{subfigure}[b]{.3\textwidth}
\centering
\includegraphics[width = \textwidth]{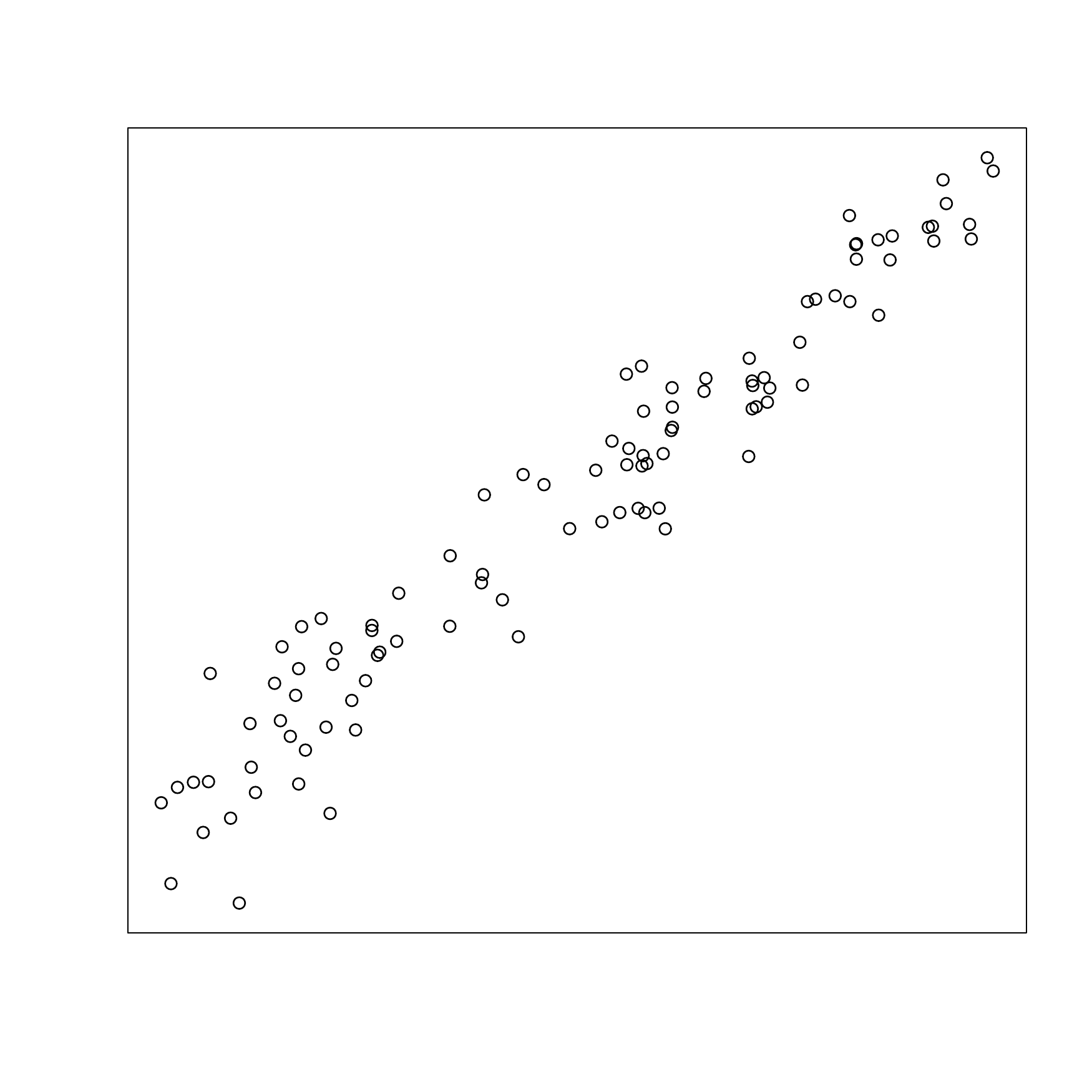}
\caption{$\xi_n =  0.732$. \label{fig0b}}
\end{subfigure}
\begin{subfigure}[b]{.3\textwidth}
\centering
\includegraphics[width = \textwidth]{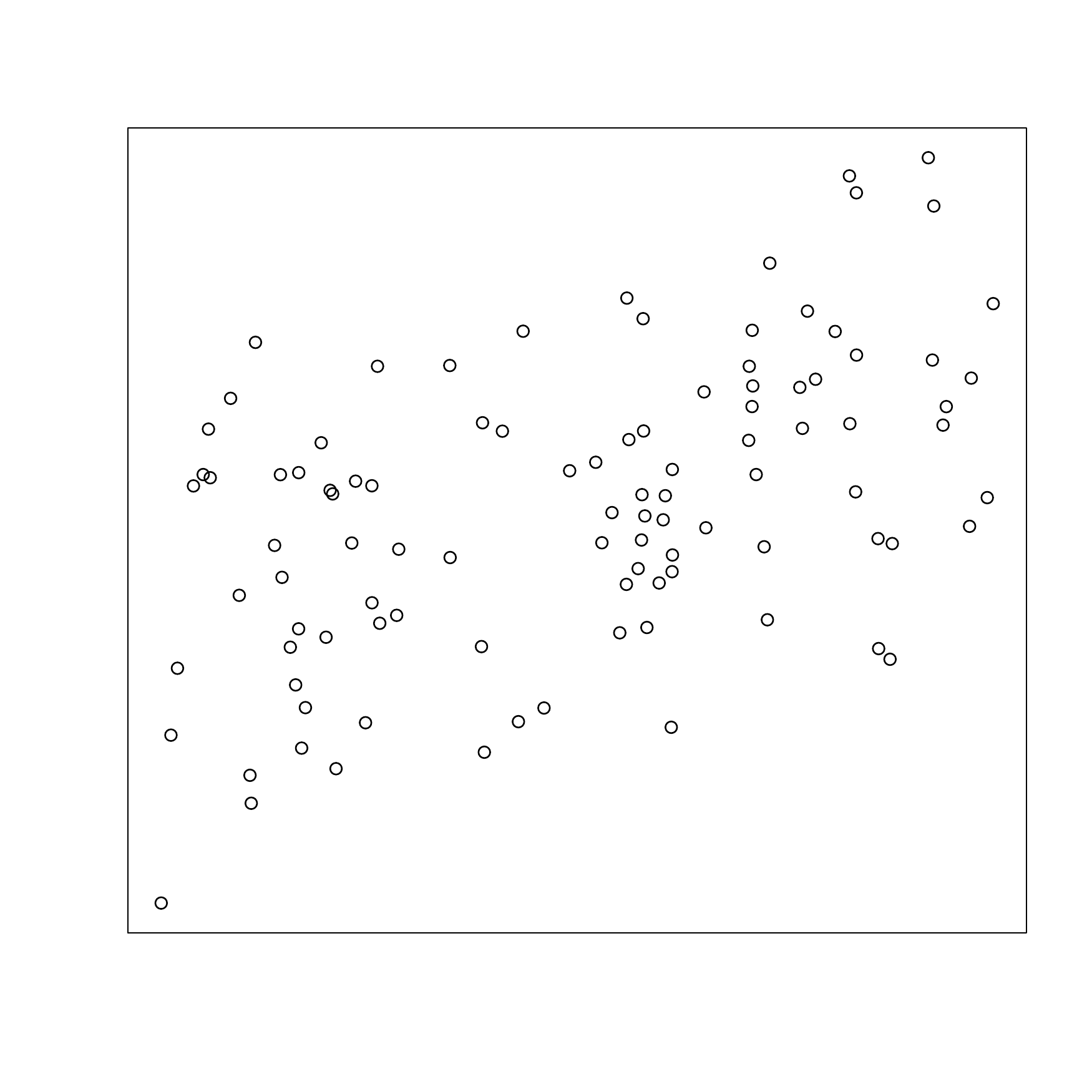}
\caption{$\xi_n =  0.145$. \label{fig0c}}
\end{subfigure}
\begin{subfigure}[b]{.3\textwidth}
\centering
\includegraphics[width = \textwidth]{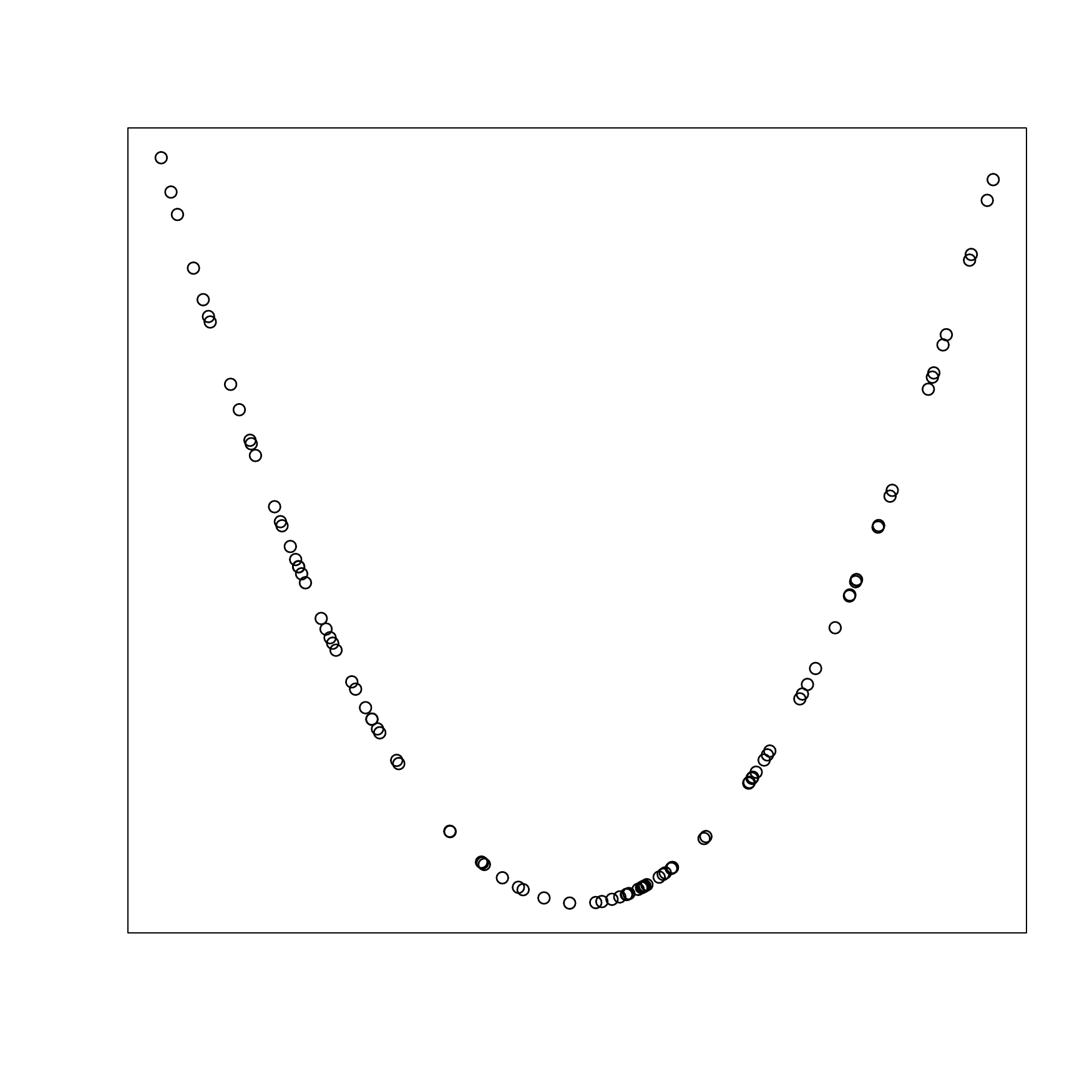}
\caption{$\xi_n = 0.941$. \label{fig0d}}
\end{subfigure}
\begin{subfigure}[b]{.3\textwidth}
\centering
\includegraphics[width = \textwidth]{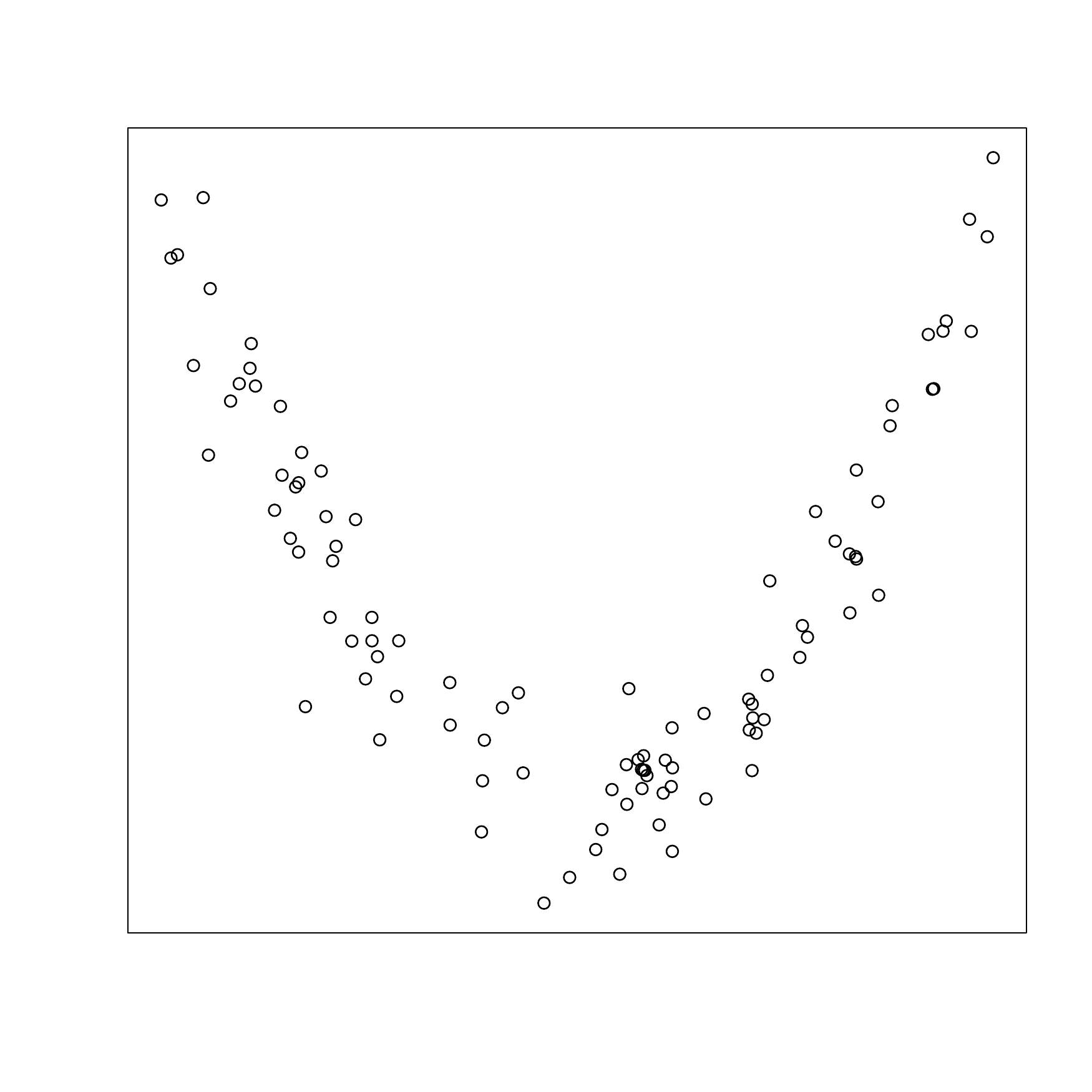}
\caption{$\xi_n = 0.684$. \label{fig0e}}
\end{subfigure}
\begin{subfigure}[b]{.3\textwidth}
\centering
\includegraphics[width = \textwidth]{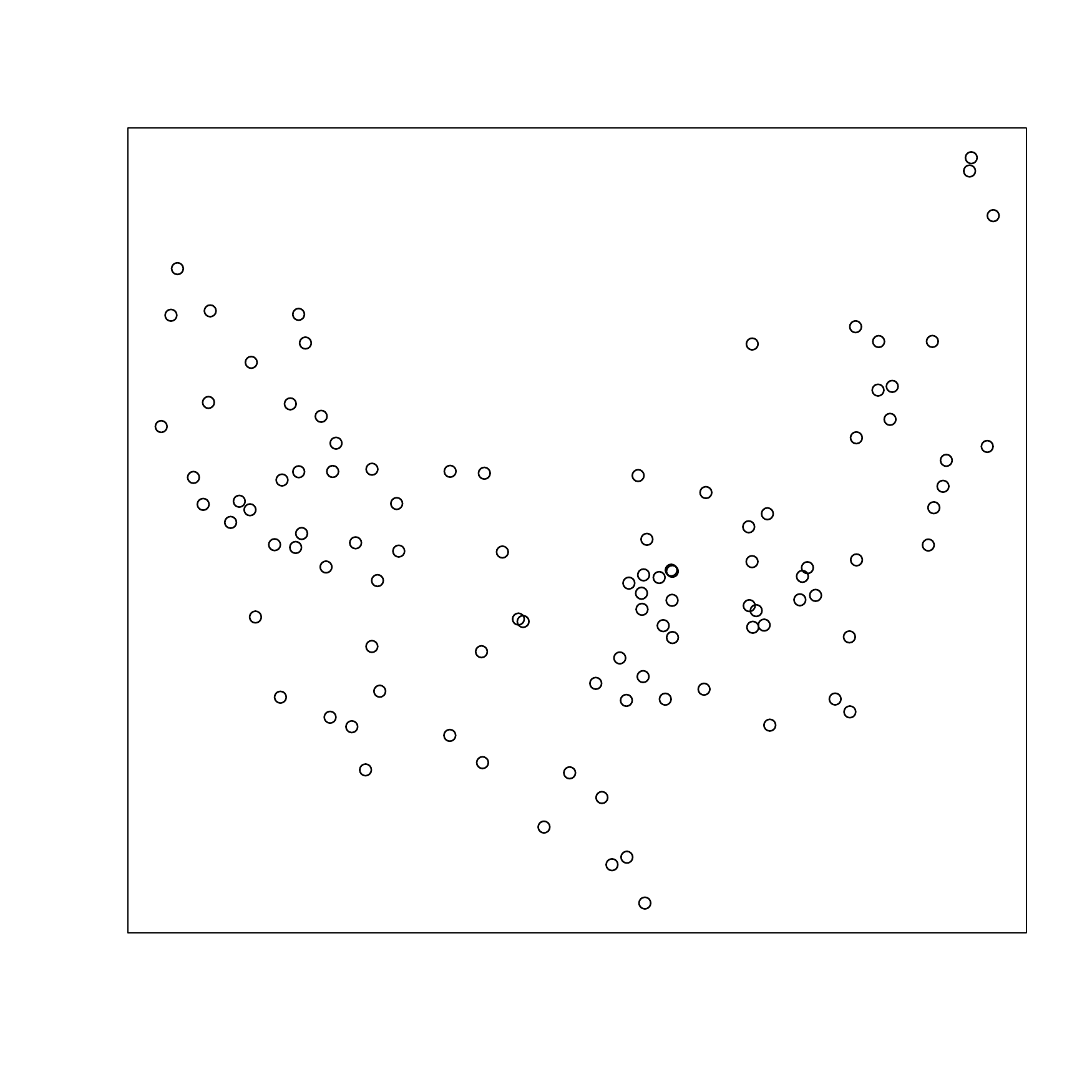}
\caption{$\xi_n = 0.265$. \label{fig0f}}
\end{subfigure}
\begin{subfigure}[b]{.3\textwidth}
\centering
\includegraphics[width = \textwidth]{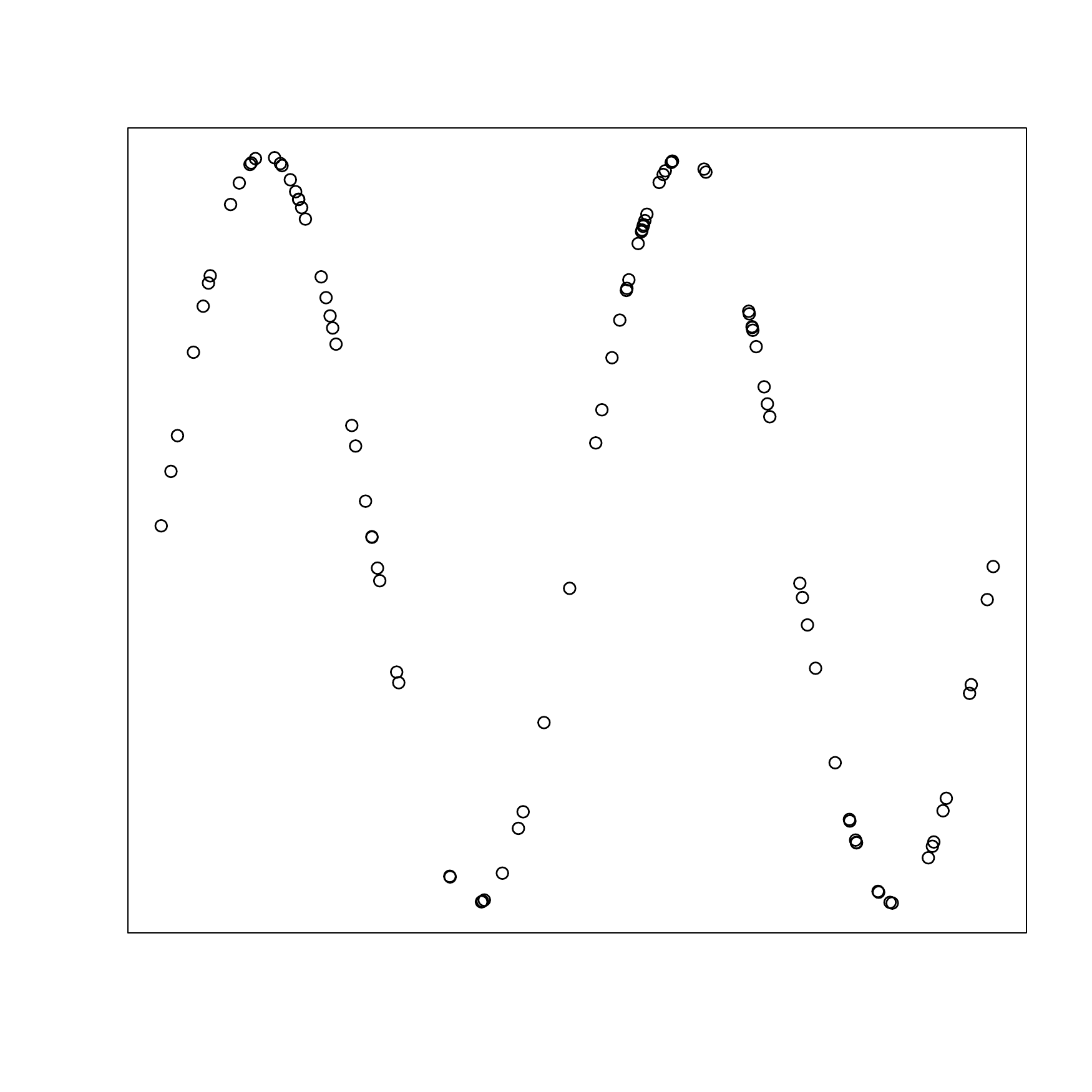}
\caption{$\xi_n = 0.885$. \label{fig0g}}
\end{subfigure}
\begin{subfigure}[b]{.3\textwidth}
\centering
\includegraphics[width = \textwidth]{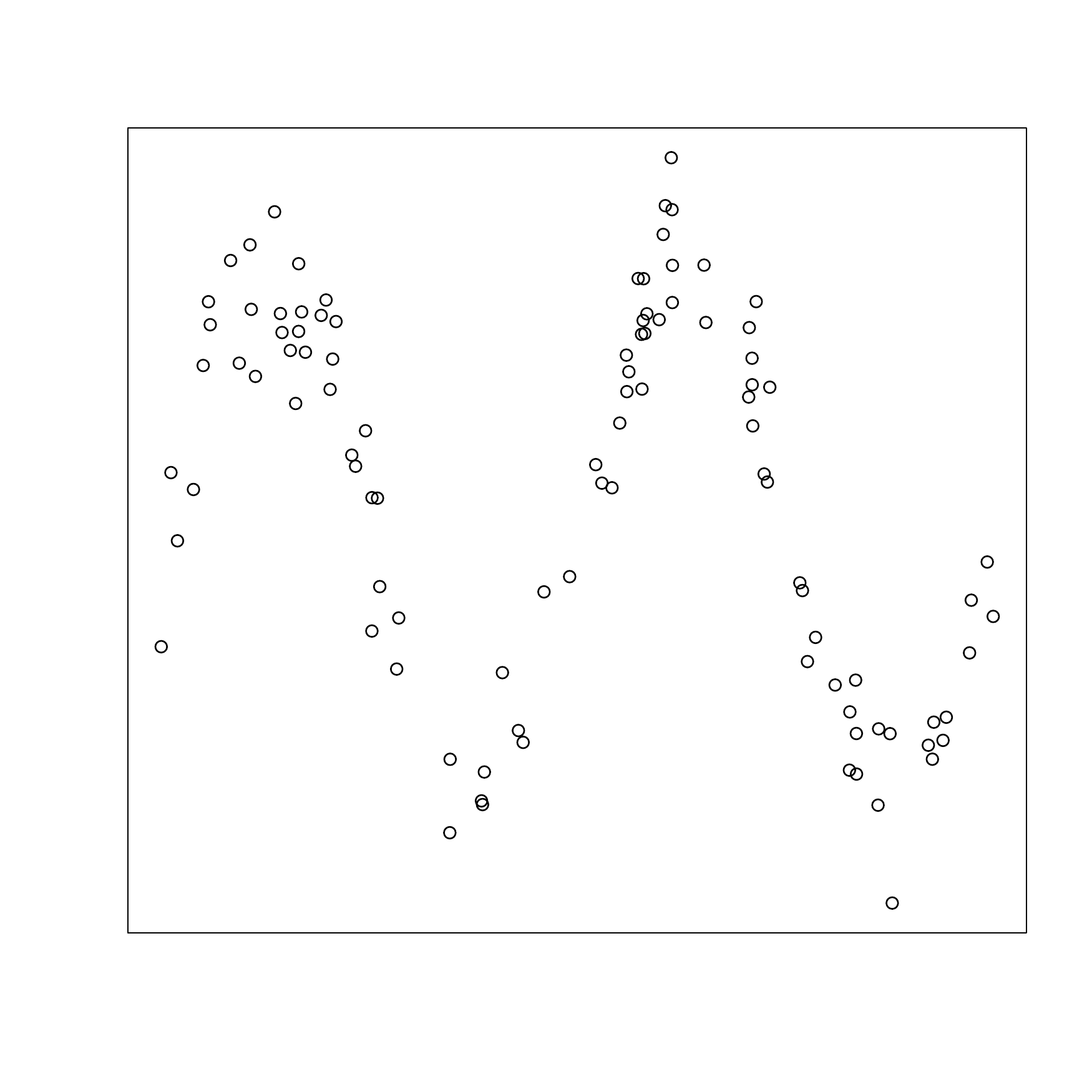}
\caption{$\xi_n = 0.650$. \label{fig0h}}
\end{subfigure}
\begin{subfigure}[b]{.3\textwidth}
\centering
\includegraphics[width = \textwidth]{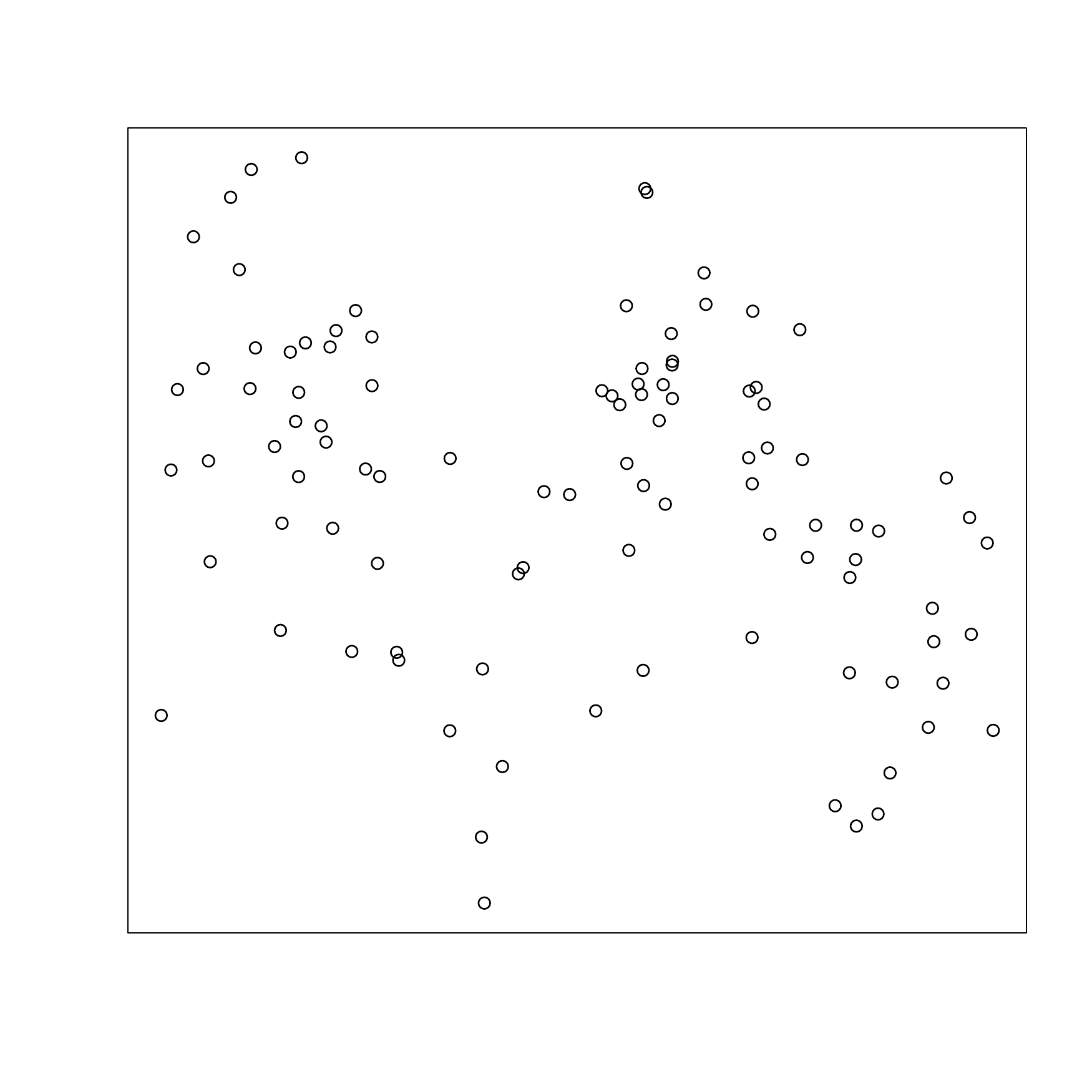}
\caption{$\xi_n= 0.281$. \label{fig0i}}
\end{subfigure}
\caption{Values of $\xi_n(X,Y)$ for various kinds of scatterplots, with $n=100$. Noise increases from left to right. The $95^{\mathrm{th}}$ percentile of $\xi_n(X,Y)$ under the hypothesis of independence is approximately $0.066$.\label{fig0}}
\end{figure}

Figure \ref{fig0} gives a glimpse of the general performance of $\xi_n$ as a measure of association. The figure has three rows. Each row starts with a scatterplot where $Y$ is a noiseless function of $X$, and $X$ is generated from the uniform distribution on $[-1,1]$. As we move to the right, more and more noise is added. The sample size $n$ is taken to be $100$ in each case, to show that $\xi_n$ performs well in relatively small samples. In each row, we see that $\xi_n(X,Y)$ is very close $1$ for the leftmost graph, and progressively deteriorates as we add more noise. By Theorem \ref{cltthm0}, the $95^{\mathrm{th}}$ percentile of $\xi_n(X,Y)$ under the hypothesis of independence, for $n=100$, is approximately $0.066$. The values in Figure \ref{fig0} are all much higher than that.

An interesting observation from Figure~\ref{fig0} is that $\xi_n$ appears to be an {\it equitable} coefficient, as defined in \cite{Reshef11}. The definition of equitability is not mathematically precise but intuitively clear. Roughly, an equitable measure of correlation `gives similar scores to equally noisy relationships of different types'. Figure \ref{fig0} indicates that $\xi_n$ has this property as long as the relationship is `functional'. It is not equitable for relationships that are not functional, although that is expected because $\xi_n$ measures how well $Y$ can be predicted by $X$. %In fact, it may be more appropriate to say that $\xi_n$ is a `prediction coefficient' than a correlation coefficient, borrowing language from~\cite{fr83}. %However, it is not equitable for `non-functional relationships' (which is not defined in a mathematically precise sense). 

The other criterion for a good measure of correlation, according to \cite{Reshef11}, is that the coefficient should be `general', in that it should be able to detect any kind of pattern in the scatterplot. In statistical terms, this means that the test of independence based on the coefficient should be consistent against all alternatives. This is clearly true by Theorem~\ref{mainthm}, in fact more true than for any other  coefficient in the literature. Among available test statistics, only maximal correlation has this property in full generality, but there is no estimator of maximal correlation that is known to be consistent for all possible distributions of $(X,Y)$.

\subsection{Validity of the asymptotic theory}\label{asympsec}

\begin{figure}[t]
\centering
\begin{subfigure}[b]{.4\textwidth}
\centering
\includegraphics[width = \textwidth]{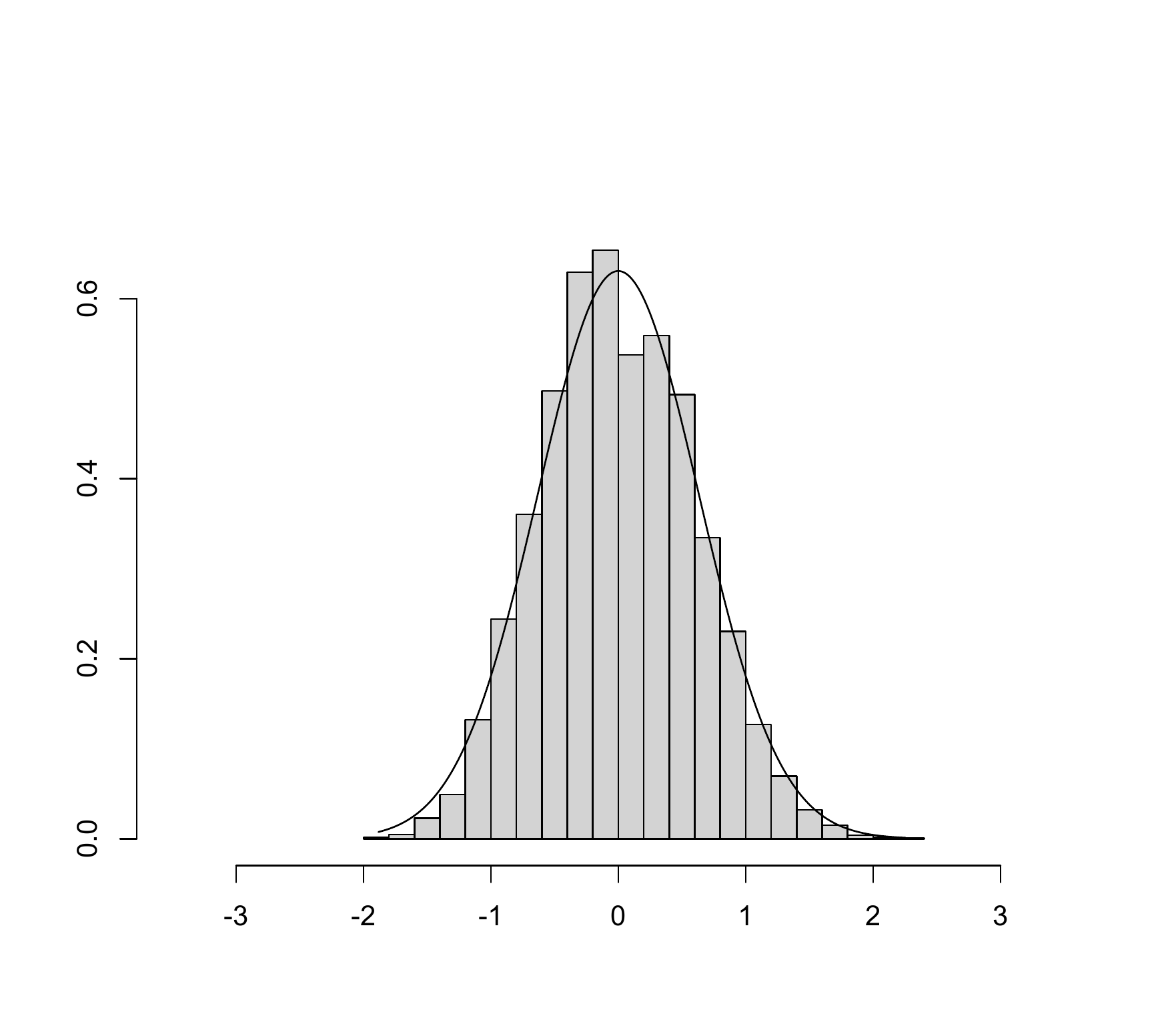}
\caption{Uniform$[0,1]$, $n=20$. \label{fig-null1}}
\end{subfigure}
\begin{subfigure}[b]{.4\textwidth}
\centering
\includegraphics[width = \textwidth]{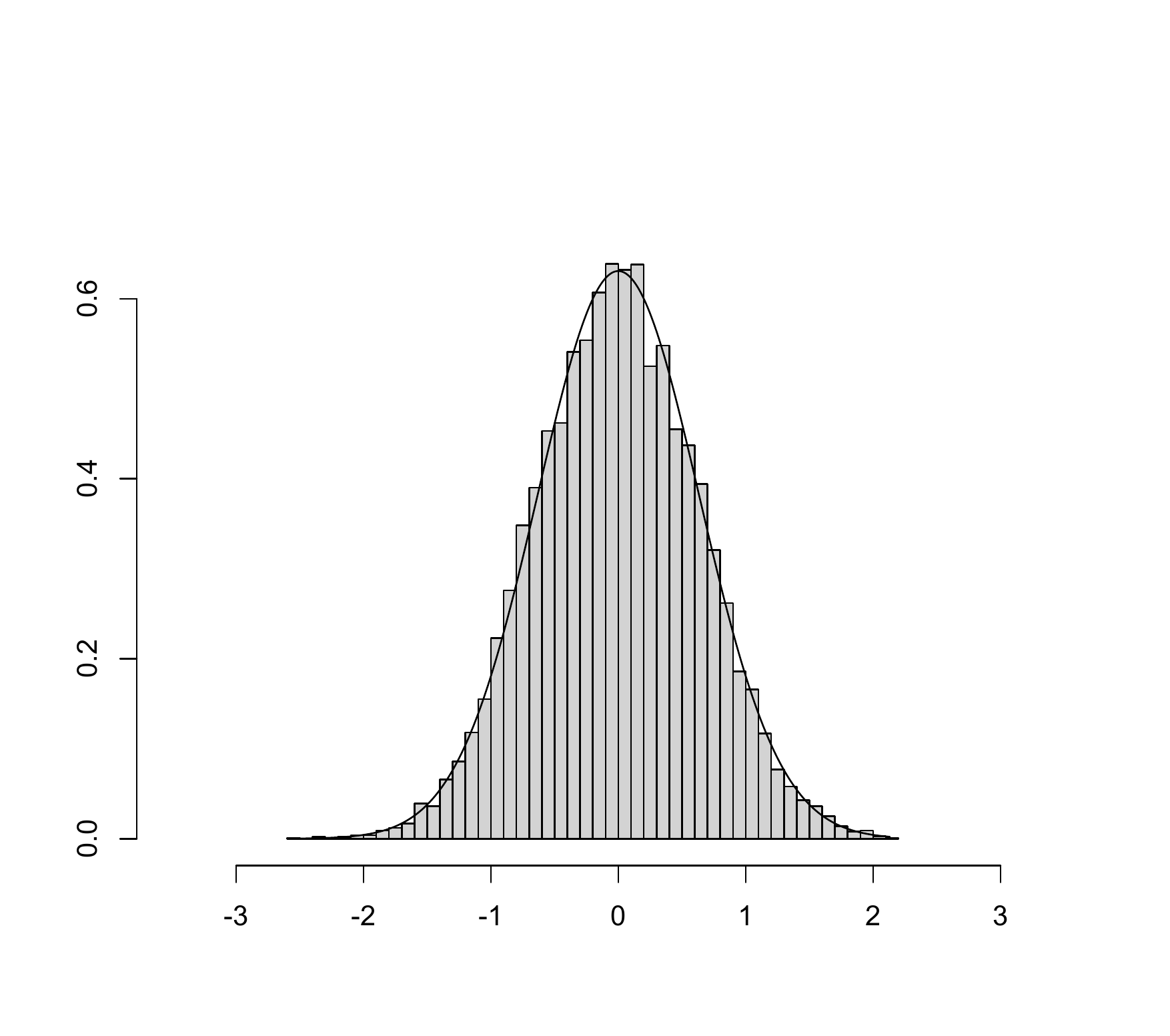}
\caption{Uniform$[0,1]$, $n=1000$. \label{fig-null2}}
\end{subfigure}
\centering
\begin{subfigure}[b]{.4\textwidth}
\centering
\includegraphics[width = \textwidth]{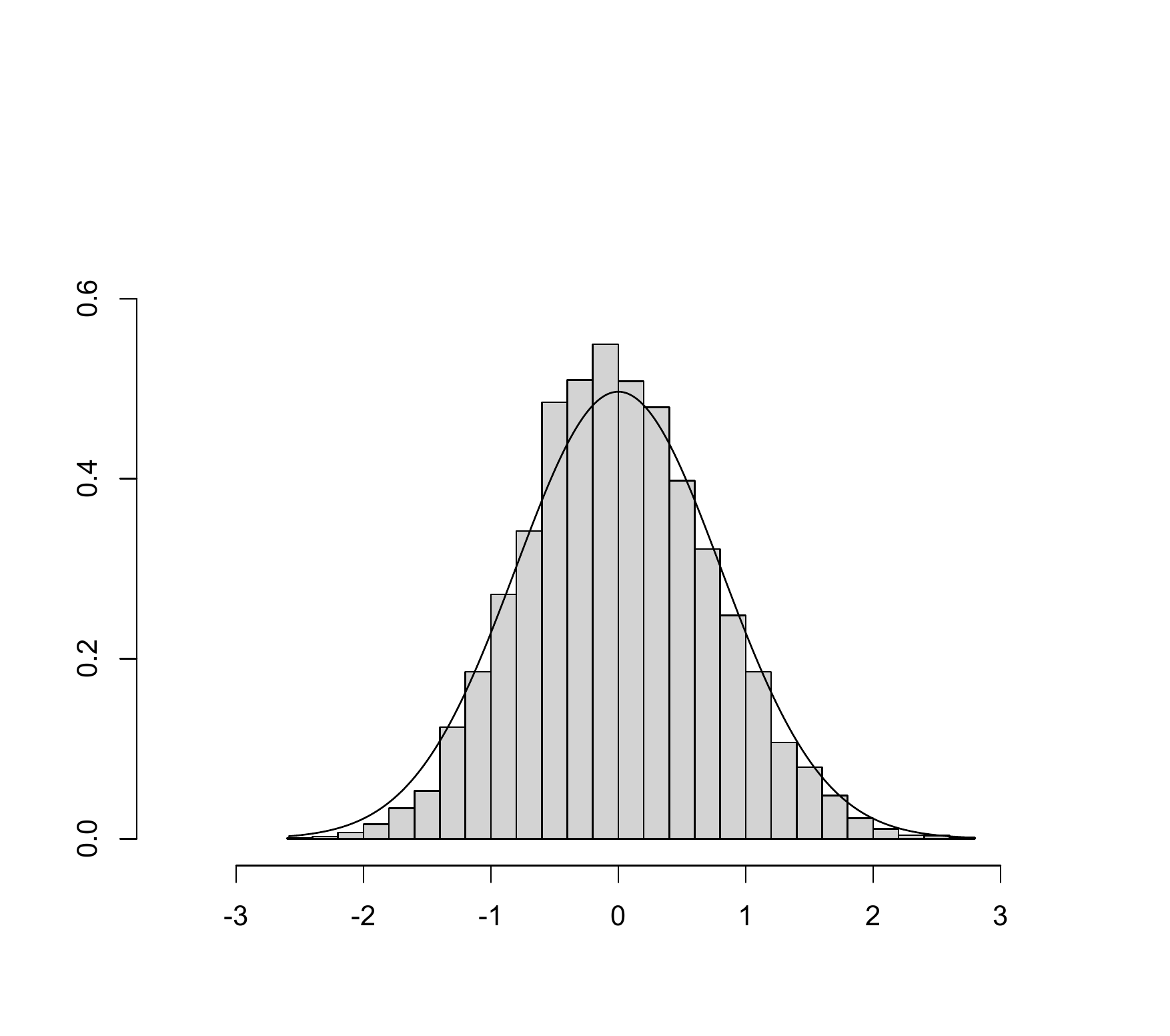}
\caption{Binomial$(3,0.5)$, $n=20$. \label{fig-null3}}
\end{subfigure}
\centering
\begin{subfigure}[b]{.4\textwidth}
\centering
\includegraphics[width = \textwidth]{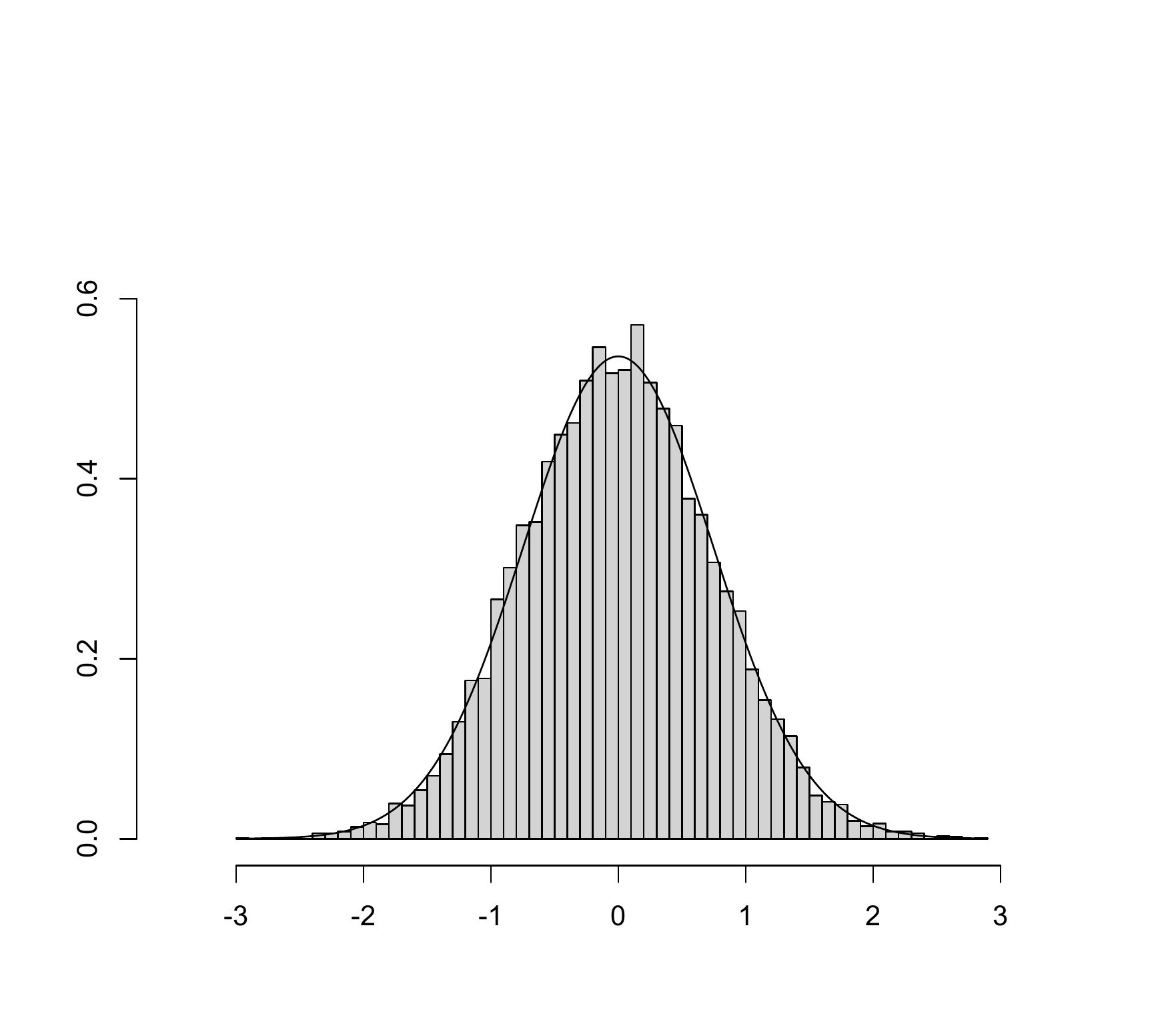}
\caption{Binomial$(3,0.5)$, $n=1000$. \label{fig-null4}}
\end{subfigure}
\caption{Histogram of ten thousand simulations of $\sqrt{n}\xi_n$, superimposed with the asymptotic density function.}
\end{figure}

Next, let us numerically investigate the distribution of $\xi_n(X,Y)$ when $X$ and $Y$ are independent. 
Taking $X_i$'s and $Y_i$'s to be independent Uniform$[0,1]$ random variables, and $n=20$, ten thousand values of $\xi_n(X,Y)$ were generated. The histogram of $\sqrt{n}\xi_n(X,Y)$ is displayed in Figure~\ref{fig-null1}, superimposed with the asymptotic density function predicted by Theorem~\ref{cltthm0}. We see that already for $n=20$, the agreement is striking. A much better agreement is obtained with $n=1000$ in Figure \ref{fig-null2}. Next, $X_i$'s and $Y_i$'s were drawn as independent Binomial$(3, 0.5)$ random variables. The value of $\tau^2$ was estimated using Theorem \ref{estthm}, and was plugged into Theorem \ref{cltthm} to obtain the asymptotic distribution of $\sqrt{n}\xi_n$. Again, the true distributions are shown to be in good agreement with the asymptotic distributions, for $n=20$ and $n=1000$,  in Figures \ref{fig-null3} and \ref{fig-null4}. 

Some simulation analysis was also carried out to investigate the convergence of $\xi_n$ under dependence. For that, the following simple model was chosen. Let $X\sim \textup{Bernoulli}(p)$ and $Z\sim\textup{Bernoulli}(p')$ be independent random variables, and let $Y:=XZ$. Then $X$ and $Y$ are  dependent Bernoulli random variables. An easy calculation shows that
\[
\xi(X,Y) = \frac{p'(1-p)}{1-pp'}.
\]
With $p = 0.4$ and $p' = 0.5$, we get $\xi(X,Y) = 0.375$. To test the convergence of $\xi_n$ to $\xi$, ten thousand simulations were carried out with $n=1000$. In this sample, the mean value of $\xi_n$ was approximately $0.374$ and the standard deviation was approximately $0.040$ (which means that the standard deviation of $\sqrt{n}\xi_n$ was approximately $1.254$). The histogram, given in Figure \ref{fig-dephist} shows an excellent fit with a normal distribution with the above mean and standard deviation.  

\begin{figure}[t]
\centering
\includegraphics[width = .5\textwidth]{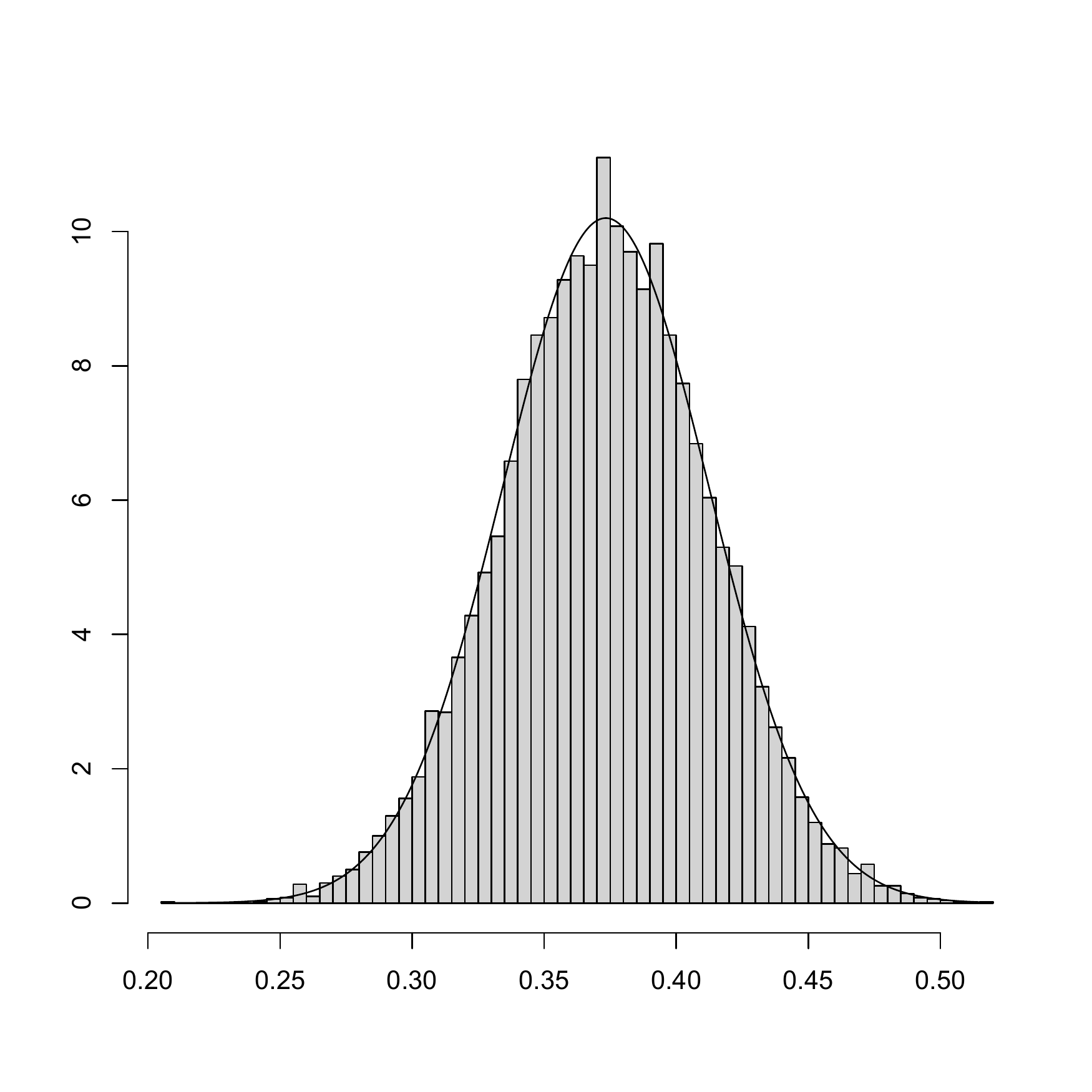}
\caption{Histogram of ten thousand simulations of $\xi_n(X,Y)$ when $X$ and $Y$ are dependent Bernoulli random variables (see Section \ref{asympsec}), superimposed with the normal density function of suitable mean and variance. Here $\xi(X,Y)=0.375$ and $n=1000$. \label{fig-dephist}}
\end{figure}

\subsection{Power and run time comparisons}\label{powersec}
In this section we compare the power of the test of independence  based on $\xi_n$ against a number of powerful tests proposed in recent years, and we also compare the run times of these tests. The main finding is that $\xi_n$ is less powerful than some of the other tests if the signal is relatively smooth, and more powerful if the signal is wiggly. In terms of run time, $\xi_n$ has a big advantage since it is computable in time $O(n\log n)$, whereas its competitors require time $n^2$. This is further validated through numerical examples, which show that $\xi_n$ is essentially the only statistic that can be computed in reasonable time if the sample size is in the order several thousands.

%The main goal of $\xi_n$ is not to test for independence, but to  estimate of the strength of the relationship between $X$ and $Y$. Still, it is necessary to understand how well it performs as a statistic for testing independence.  %The upshot is that theoretical guarantees of optimality do not have much bearing on practical performance, since uniform smoothness as $n\to\infty$ is not a very practical criterion. Indeed, simulation studies seem to be the only way of understanding the true power of independence tests for various kinds of alternatives. 

Comparisons are carried out with the following popular test statistics for testing independence.  I excluded  statistics that are either too new (because they are not time-tested, and software is not available in many cases) or too old (because they are superseded by newer ones). In the following, $(X_1,Y_1),\ldots, (X_n, Y_n)$ is an i.i.d.~sample of points from some distribution on~$\rr^2$. 
\begin{enumerate}
\item \textit{Maximal information coefficient (MIC) \cite{Reshef11}:} Recall that the mutual information of a  bivariate probability distribution is the Kullback--Leibler divergence between that distribution and the product of its marginals. Given any scatterplot of $n$ points, suppose we divide it into an $x\times y$ array of rectangles. The proportions of points falling into these rectangles define a bivariate probability distribution. Let $I$ be the mutual information of this probability distribution. The maximum of $I/\log \min\{x,y\}$ over all subdivisions into rectangles, under the constraint $xy< n^{0.6}$, is called the maximal information coefficient of the scatterplot. 
\item \textit{Distance correlation~\cite{Szekely07}:} Let $a_{ij} := |X_i-X_j|$ and $b_{ij} := |Y_i-Y_j|$. Center these numbers by defining $A_{ij} := a_{ij}-a_{i\cdot} - a_{\cdot j} + a_{\cdot\cdot}$ and $B_{ij} := b_{ij}-b_{i\cdot} - b_{\cdot j} + b_{\cdot\cdot}$, where $a_{i\cdot}$ is the average of $a_{ij}$ over all $j$, etc. The distance correlation between the two samples is simply the Pearson correlation between the $A_{ij}$'s and the $B_{ij}$'s. %Surprisingly, this statistic gives a consistent test for independence under some mild conditions.
\item \textit{The HHG test~\cite{hhg13}:} Take any $i$ and $j$. Divide $X_k$'s into two groups depending on whether $|X_i-X_k|< |X_i-X_j|$ or not. Similarly classify the $Y_k$'s into two groups depending on whether $|Y_i-Y_k|< |Y_i-Y_j|$ or not. These classifications partition the scatterplot into $4$ compartments, and the numbers of points in these compartments define a $2\times 2$ contingency table. The HHG test statistic is a linear combination of the Pearson $\chi^2$ statistics for testing independence in these contingency tables over all choices of $i$ and $j$. %Under mild conditions on the law of $(X,Y)$, this statistic gives a consistent test for independence.
\item \textit{The Hilbert--Schmidt independence criterion (HSIC)~\cite{gretton05, gretton08}:} Let $k$ and $l$ be symmetric positive definite kernels on $\rr^2$. For example, we may take the Gaussian kernel $k(x,y) = l(x,y)= e^{-|x-y|^2/2\sigma^2}$ for some $\sigma>0$. Let $k_{ij} := k(X_i, X_j)$ and $l_{ij} := l(Y_i,Y_j)$. Then the  HSIC statistic is 
\[
\frac{1}{n^2}\sum_{i,j} k_{ij} l_{ij} + \frac{1}{n^4} \sum_{i,j,q,r} k_{ij} l_{qr} - \frac{2}{m^3}\sum_{i,j,q} k_{ij} l_{iq}. 
\]
% (building on earlier works of~\cite{ingster93a, ingster93b, ingster93c, ls99}). 
\end{enumerate}
All of the above test statistics are consistent for testing independence under mild conditions. Moreover, the HSIC test has been proved to be minimax rate-optimal against uniformly smooth alternatives~\cite{liyuan19}.
%\vskip.2in

\begin{figure}[t]
\includegraphics[width = .95\textwidth]{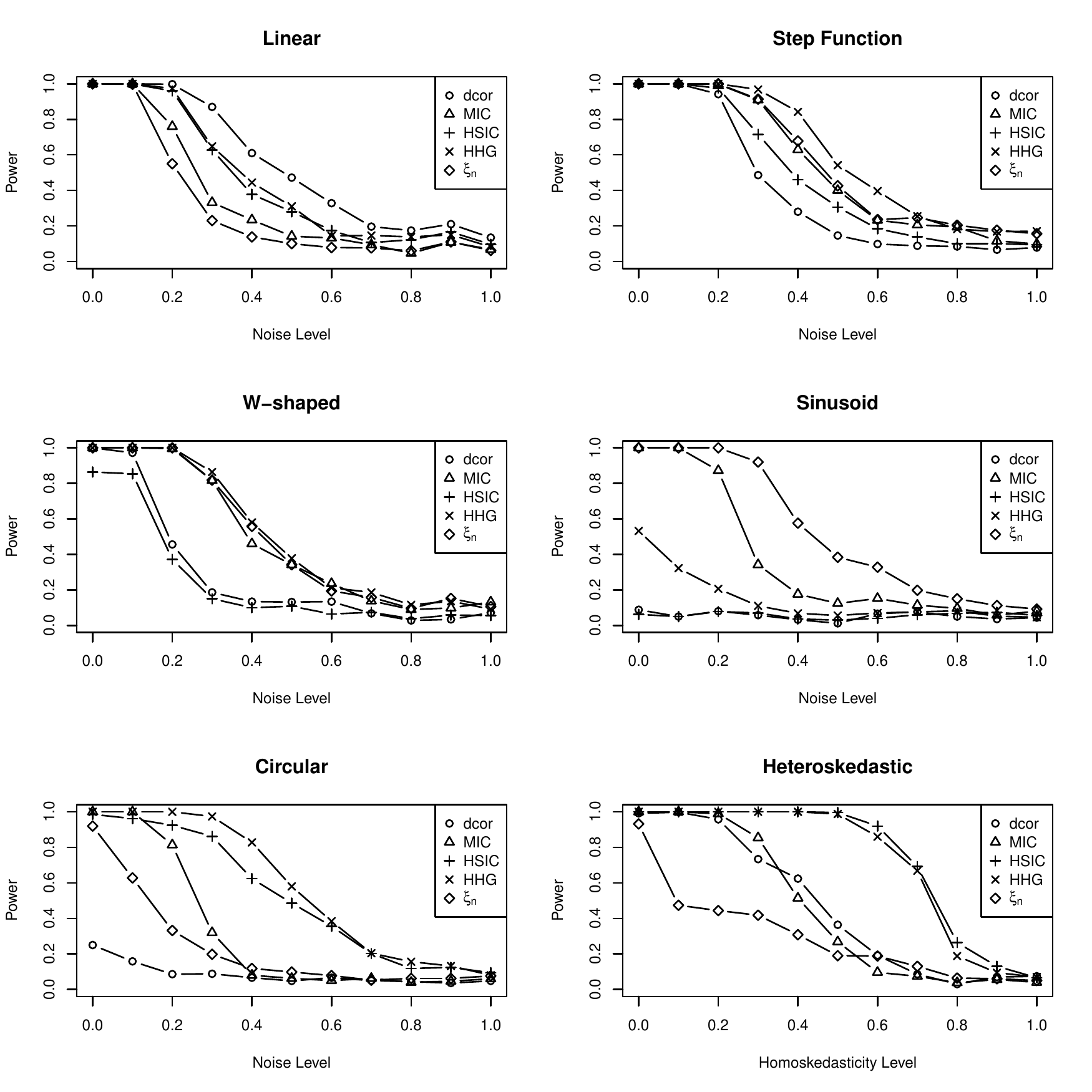}
\caption{Comparison of powers of several tests of independence. The titles describe the shapes of the scatterplots. The level of the noise increases from left to right. In each case, the sample size is $100$, and $500$ simulations were used  to estimate the power.\label{fig1}}
\end{figure}

Power comparisons were carried out with sample size $n = 100$. In each case, $500$ simulations were used to estimate the power. The R packages energy, minerva, HHG and dHSIC were used for calculating the distance correlation, MIC, HHG and HSIC statistics, respectively. Since the HHG test is very slow for large samples, a fast univariate version of the HHG test~\cite{hhg2} was used. Generating $X$ from the uniform distribution on $[-1,1]$, the following six alternatives were considered:
\begin{enumerate}
\item Linear: $Y = 0.5X + 3\lambda \ve$, where $\lambda$ is a noise parameter ranging from $0$ to $1$, and $\ve \sim N(0,1)$ is independent of $X$. 
\item Step function: $Y = f(X) +10 \lambda \ve$, where $f$ takes values $-3$, $2$, $-4$ and $-3$ in the intervals $[-1,-0.5)$, $[-0.5, 0)$, $[0,0.5)$ and $[0.5,1]$.
\item W-shaped: $Y = |X+0.5|1_{\{X<0\}} + |X-0.5|1_{\{X\ge 0\}} + 0.75\lambda \ve$. 
\item Sinusoid: $Y = \cos 8\pi X + 3\lambda \ve$. 
\item Circular: $Y = Z\sqrt{1-X^2} + 0.9 \lambda \ve$, where $Z$ is $1$ or $-1$ with equal probability, independent of $X$.
\item Heteroskedastic: $Y = 3(\sigma(X) (1-\lambda)+\lambda)\ve$, where $\sigma(X) = 1$ if $|X|\le 0.5$ and $0$ otherwise. As $\lambda$ increases from $0$ to $1$, the relationship becomes more and more homoskedastic. 
\end{enumerate}
The coefficients in all of the above were chosen to ensure that a full range of powers were observed as $\lambda$ was varied from $0$ to $1$. The results are presented in Figure~\ref{fig1}. The main observation from this figure is that $\xi_n$ is more powerful than the other tests when the signal has an oscillatory nature, such as for the W-shaped scatterplot and the sinusoid. For the step function, too, it performs reasonably well. However, $\xi_n$ has inferior performance for smoother alternatives, namely, the linear, circular, and heteroskedastic scatterplots.

\begin{table}[t]
\renewcommand\arraystretch{1.2}
\begin{center}
\begin{footnotesize}
\caption{Run times (in seconds) for permutation tests of independence, with $200$ permutations. For $\xi_n$, the asymptotic test was used because it is as reliable as the permutation test.\label{permrun}}
\begin{tabular}{l>{\raggedleft}p{0.09\linewidth}>{\raggedleft}p{0.11\linewidth}>{\raggedleft}p{0.09\linewidth}>{\raggedleft}p{0.12\linewidth}>{\raggedleft\arraybackslash}p{0.07\linewidth}}
\toprule
$n$ & dCor & MIC & HSIC & HHG  & $\xi_n$ \\
\midrule
100 & 0.008 & 0.328 & 0.048  & 0.167 &  0.006  \\
500 & 0.104  & 5.433 & 1.214 & 4.671 & 0.007 \\
1000 & 0.532 & 17.459  & 5.028 &  20.515 & 0.009  \\ 
2000 &  2.423 & 55.556 & 18.873 & 108.949 &  0.009\\
10000 &88.976 & 1097.483 & 860.605 & $>30$ mins & 0.011 \\
\bottomrule
\end{tabular}
\end{footnotesize}
\end{center}
\end{table}

Next, let us turn to the comparison of run times for tests of independence based on the five competing test statistics. For all except $\xi_n$, the only way to test for independence is to run a permutation test. (There is a theoretical test for HSIC, but it is only a crude approximation.) The number of permutations was taken to be the smallest respectable number, $200$. Usually $200$ is too small for a permutation test, but I took it to be so small so that the program terminates in a manageable amount of time for the larger values of~$n$. For $\xi_n$, the asymptotic test was used because it performs as well as the permutation test even in very small samples, as we saw in Section~\ref{asympsec}.% (and I have also verified that in further simulations). 

For distance correlation, HSIC and HHG, the permutation tests are directly available from the corresponding R packages. For MIC, I had to write the code because the permutation tests are not automatically available from the package, so the run time can probably be somewhat improved with a better code. For the HHG test, the function requires the distance matrices for $X$ and $Y$ to be input as arguments. For the sake of fairness, the time required for computing the distance matrices was included in the total time for carrying out the permutation tests. 

The results are presented in Table~\ref{permrun}. Every test was hundreds or even thousands of times slower than the test based on $\xi_n$ for all sample sizes $500$ and above. For sample size $10000$, the HHG test was terminated after not converging in 30 minutes.

\section{Example: Yeast gene expression data}\label{spellmansec}

\begin{figure}[t]
\includegraphics[width = .9\textwidth]{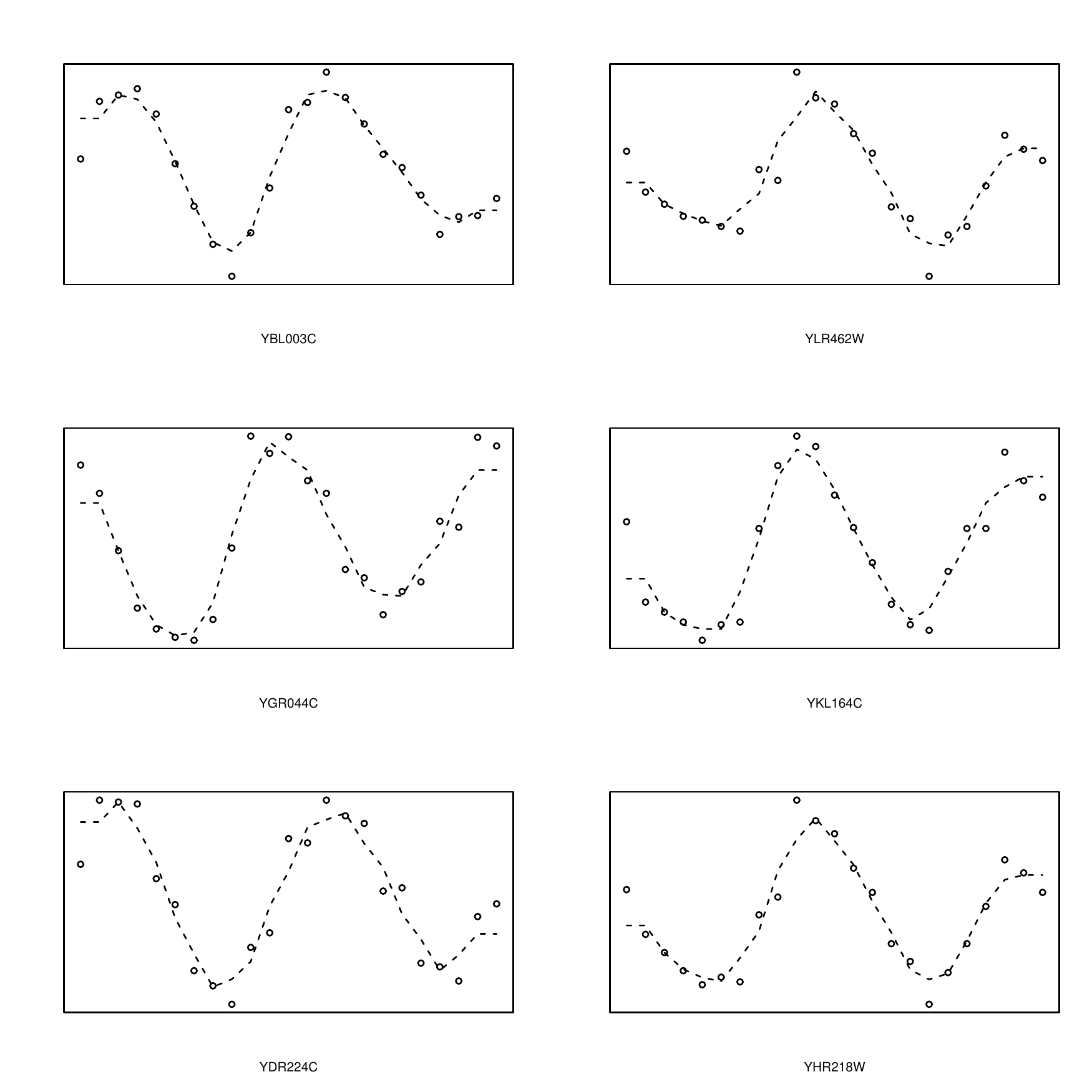}
\caption{Transcript levels of the top 6 among the 215 genes selected by $\xi_n$ but by no other test. The dashed lines are fitted by $k$-nearest neighbor regression with $k=3$. The name of the gene is displayed below each plot. \label{spellfig1}}
\end{figure}

In a landmark paper in gene expression studies~\cite{spellman98}, the authors studied the expressions of $6223$ yeast genes with the goal of identifying genes whose transcript levels oscillate during the cell cycle. In lay terms, this means that the expressions were studied over a number of successive time points ($23$, to be precise), and the goal was to identify the genes for which the transcript levels follow an oscillatory pattern. This example illustrates the utility of correlation coefficients in detecting patterns, because the number of genes is so large that identifying patterns by visual inspection is out of the question.

\begin{figure}[t]
\includegraphics[width = .9\textwidth]{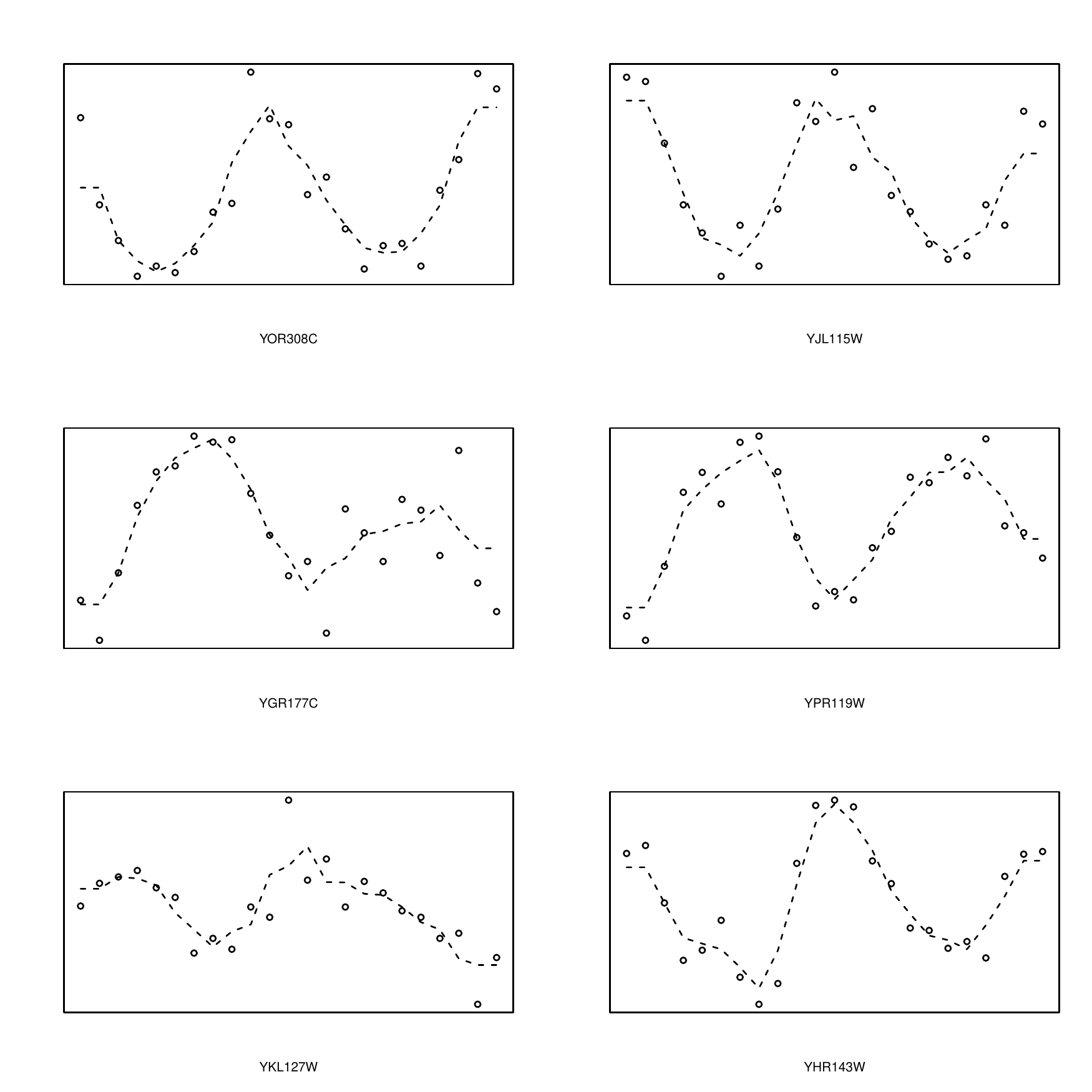}
\caption{Transcript levels of a random sample of 6 genes from the 215 genes that were selected by $\xi_n$ but by no other test.\label{spellfig3}}
\end{figure}

This dataset was used in the paper~\cite{Reshef11} to demonstrate the efficacy of MIC for identifying patterns in scatterplots. The authors of \cite{Reshef11} used a curated version of the dataset, where they excluded all genes for which there were missing observations, and made several other modifications. The revised dataset has $4381$ genes. I used this curated dataset (available through the R package minerva) to study the power of $\xi_n$ in discovering genes with oscillating transcript levels, and compare its performance with the competing tests from Section \ref{powersec}. 

There are literally hundreds of papers analyzing this particular dataset. I will not attempt to go deep into this territory in any way, because that will take us too far afield. The sole purpose of the analysis that follows is to compare the performance of $\xi_n$ with the  competing tests.

\begin{figure}[t]
\includegraphics[width = .9\textwidth]{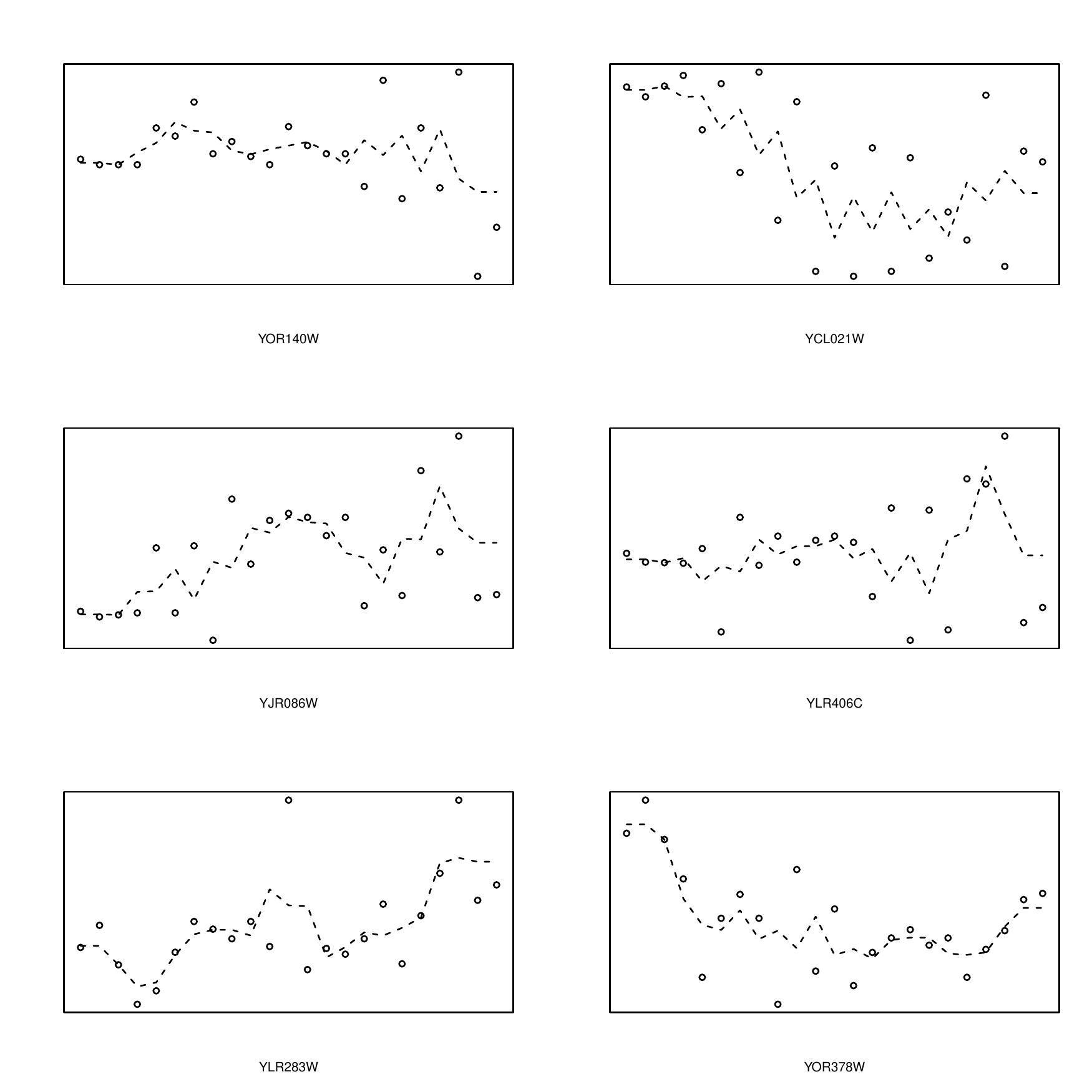}
\caption{Transcript levels of 6 randomly sampled genes from the set of genes that were not selected by $\xi_n$ but were selected by at least one other test.\label{spellfig4}}
\end{figure}

For each test, P-values were obtained and a set of significant genes were selected using the Benjamini--Hochberg FDR procedure~\cite{bh95}, with the expected proportion of false discoveries set at $0.05$.

It turned out that {\it there are 215 genes (out of 4381) that are selected by $\xi_n$ but by none of the other  tests.} This is surprising in itself, but what is more surprising is the nature of these genes. Figure \ref{spellfig1} shows the transcript levels of the top $6$ of these genes (that is, those with the smallest P-values). There is no question that these genes exhibit almost perfect oscillatory behavior and yet they were not selected by any of the  other tests.

One may wonder if this is true for only the top 6 genes, or typical of all 215. To investigate that, I took a random sample of 6 genes from the 215, and looked at their transcript levels. The results are shown in Figure \ref{spellfig3}. Even for a random sample, we see strong oscillatory behavior. This behavior was consistently observed in other random samples.

How about the genes that were selected by at least one of the other tests, but not by $\xi_n$? Figure \ref{spellfig4} shows the transcript levels of a random sample of 6 genes selected from this set. I think it is reasonable to say that these plots show slight increasing or decreasing trends, or heteroscedasticity, but no definite oscillatory patterns. Repeated samplings showed similar results. 

Thus, we arrive at the following conclusion. The genes selected by $\xi_n$ are much more likely than the genes selected by the other tests to be the ones that really exhibit oscillatory patterns in their transcript levels during the cell cycle. This is because the other tests prioritize monotone trends over cyclical patterns. Most of the 215 genes that were selected by $\xi_n$ but not by any of the other tests show pronounced oscillatory patterns.   The fact that $\xi_n$ is particularly powerful for detecting oscillatory behavior turns out to be very useful in this example. Of course, $\xi_n$ also selects genes that show other kinds of patterns (it selects a total of 586 genes), but those are  selected by at least one of the other tests and therefore do not appear in this set of 215 genes that are selected exclusively by $\xi_n$.

\begin{figure}[t]
\includegraphics[width = .55\textwidth]{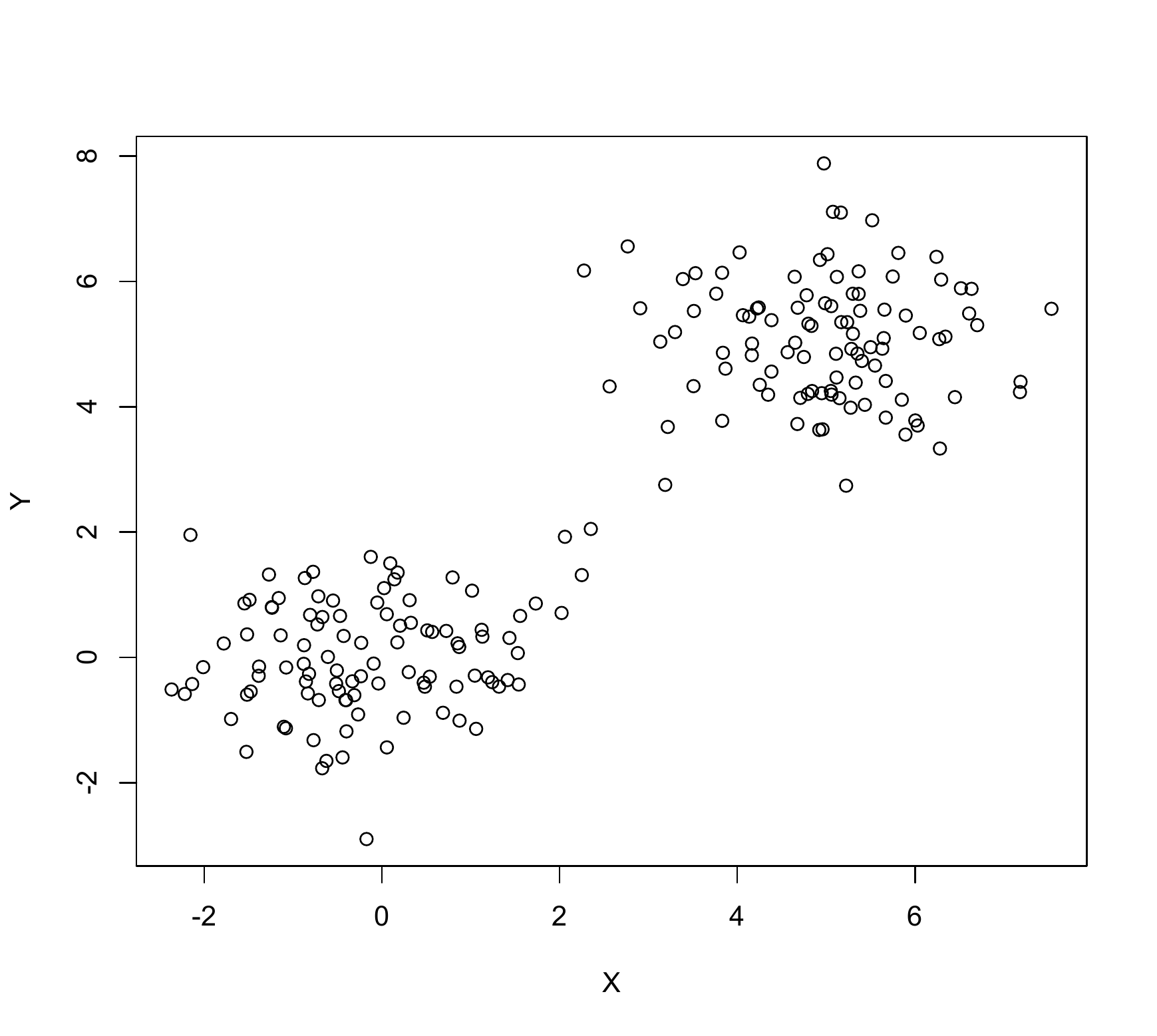}
\caption{Scatterplot of a mixture of bivariate normals, with  $n=200$. For this plot, maximal correlation $=0.99$, MIC $=1.00$, and $\xi_n= 0.48$. \label{figx2}}
\end{figure}

\section{MIC and maximal correlation may not correctly measure the strength of the relationship}\label{micsec}

It is sometimes mistakenly believed that MIC and maximal correlation measure the strength of relationship between $X$ and $Y$; in particular, that they attain their maximum value, $1$, if and only if the relationship between $X$ and $Y$ is perfectly noiseless. In this section we show that this is not true:  MIC and maximal correlation can detect noiseless relationships even if the actual relationship between $X$ and $Y$ is very noisy.

In the example shown in Figure \ref{figx2}, $200$ samples of $(X,Y)$ are generated from a mixture of bivariate normal distributions. With probability $1/2$, $(X,Y)$ is drawn from the standard bivariate normal distribution, and with probability $1/2$, $(X,Y)$ is drawn from the bivariate normal distribution with mean $(5,5)$ and identity covariance matrix. The data forms two clusters of roughly equal size that are close but nearly disjoint. Clearly, there is a lot of noise in the relationship between $X$ and $Y$. Given $X$, we can only tell whether $Y$ comes from $N(0,1)$ or $N(5,1)$, but nothing else. Yet, rounded off to two decimal places, MIC is $1.00$ and maximal correlation (as computed by the ACE algorithm~\cite{Breiman85}) is $0.99$ for this scatterplot. The coefficient $\xi_n$, on the other hand, is well-behaved; it turns out to be $0.48$, indicating the presence of a significant relationship between $X$ and $Y$ but not a noiseless one. Common sense suggests that the value $0.48$ is much better reflective of the strength of the relationship between $X$ and $Y$ in Figure \ref{figx2} than $0.99$ or~$1.00$.

In the supplementary material of~\cite{Reshef11}, it is shown that MIC $=1$ when $Y = f(X)$ for a large class of functions $f$. However, it is not shown that the {\it converse} is true, that is MIC $=1$ implies that $X$ and $Y$ have a noiseless relationship. Figure \ref{figx2} indicates that in fact the converse is probably {\it not true}. The phenomenon is not an artifact of the sample size --- it remains consistently true in larger sample sizes. Moreover, scatterplots such as  Figure \ref{figx2} are not uncommon in real datasets. 

The following mathematical result uses the intuition gained from the above example to confirm that there indeed exist very noisy relationships which are declared to be perfectly noiseless by maximal correlation and MIC.
\begin{prop}\label{micthm}
Let $I_1$, $I_2$, $J_1$ and $J_2$ be bounded intervals such that $I_1$ and $I_2$ are disjoint, and $J_1$ and $J_2$ are disjoint. Suppose that the law of a random vector $(X,Y)$ is supported on the union of the two rectangles $I_1\times J_1$ and $I_2\times J_2$, giving equal masses to both. Then the maximal correlation between $X$ and $Y$ is $1$, and the MIC between $X$ and $Y$ in an i.i.d.~sample of size $n$ tends to $1$ in probability as $n\to \infty$. 
\end{prop}
\begin{proof}
Recall that the maximal correlation between two random variables $X$ and $Y$ is defined as the maximum possible correlation between $f(X)$ and $g(Y)$ over all $f$ and $g$ such that $f(X)$ and $g(Y)$ are square-integrable. In the setting of this proposition, let $f$ be the indicator of the interval $I_1$ and $g$ be the indicator of the interval $J_1$. Then $f(X) = 1$ if and only if $g(Y)=1$, because the nature of $(X,Y)$ implies that $X\in I_1$ if and only if $Y\in J_1$. Thus, $f(X)=g(Y)$, and so the maximal correlation between $X$ and $Y$ is equal to $1$. 

Next, recall the definition of MIC from Section \ref{powersec}. The support of $(X,Y)$ can be partitioned into the $2\times 2$ array of rectangles $I_1\times J_1$, $I_1\times J_2$, $I_2\times J_1$ and $I_2\times J_2$. The first and fourth rectangles carry mass $1/2$ each, and the other two carry mass $0$. Therefore, when $n$ is large, the first and fourth rectangles receive approximately $n/2$ points each, and the other two receive no points. A simple calculation shows that the mutual information of the corresponding contingency table is approximately $\log 2$. Thus, the contribution of this array of rectangles to the definition of MIC is approximately $1$, which shows that the MIC itself is approximately $1$ (since it cannot exceed $1$ and is defined to be the maximum of the contributions from all rectangular partitions of size $< n^{0.6}$). 
\end{proof}

\section{Summary}\label{adsec}
Let us now briefly summarize what we learned. The new correlation coefficient offers many advantages over its competitors. The following is a partial list:
\begin{enumerate}
\item It has a very simple formula. The formula is as simple as those for the classical coefficients, like Pearson's correlation, Spearman's $\rho$, or Kendall's $\tau$.
\item Due to its simple formula, it is (a) easy to understand conceptually, and (b) computable very quickly, not only in theory but also in practice. Most of its competitors are hundreds of times slower to compute even in samples of moderately large size, such as $500$.
\item It is a function of ranks, which makes it robust to outliers and invariant under monotone transformations of the data. 
\item It converges to a limit which has an easy interpretation as a measure of dependence. The limit ranges from $0$ to $1$. It is $1$ if and only if $Y$ is a measurable function of $X$ and $0$ if and only if $X$ and $Y$ are independent. Thus, $\xi_n$ gives an actual measure of the strength of the relationship. %It is the only coefficient in the literature that has this property. 
\item It has a very simple asymptotic theory under the hypothesis of independence, which is roughly valid even for samples of  size as small as $20$. This allows theoretical tests of independence, bypassing computationally expensive permutation tests that are necessary for other tests.
\item The test of independence based on $\xi_n$ is consistent against all alternatives, with no exceptions. No other test has this property.
\item None of the results mentioned above require any assumptions about the law of $(X,Y)$ except that $Y$ is not a constant. One can even apply $\xi_n$ to categorical data, by converting the categorical variables to integer-valued variables in any arbitrary way.
\item In simulations and real data, $\xi_n$ seems to be more powerful than other tests for detecting oscillatory signals.
\end{enumerate}
Against all of the above advantages, $\xi_n$ has only one disadvantage: It seems to have  less power than several popular tests of independence when the signal is smooth and non-oscillatory. Although such signals comprise the majority of types observed in practice, this is a matter of concern only when the sample size is small. In large samples, all tests are powerful, and computational time becomes a much bigger concern. %Even for moderately large samples, the many advantages of $\xi_n$ trump this one disadvantage. 

\section{Proof sketch}\label{sketchsec}
This section contains a brief sketch of the proof of convergence of $\xi_n$ to $\xi$. For simplicity, let us only consider the case of continuous $X$ and $Y$.  First, note that by the Glivenko--Cantelli theorem, $r_i/n \approx F(Y_{(i)})$, where $F$ is the cumulative distribution function of $Y$. Thus, 
\begin{align}\label{xiapprox}
\xi_n(X,Y) \approx 1 - \frac{3}{n} \sum_{i=1}^n |F(Y_i)- F(Y_{N(i)})|,
\end{align}
where $N(i)$ is the unique index $j$ such that $X_j$ is immediately to the right of $X_i$ if we arrange the $X$'s in increasing order. If $X_i$ is the rightmost value, define $N(i)$ arbitrarily; it does not matter since the contribution of a single term in the above sum is $O(1/n)$.

The first important observation is that for any $x,y\in \rr$, 
\begin{align}\label{fform}
|F(x)- F(y)| &= \int (1_{\{t\le x\}} - 1_{\{t\le y\}})^2 d\mu(t),
\end{align}
where $\mu$ is the law of $Y$. This is true because the integrand is $1$ between $x$ and $y$ and $0$ outside.

Now suppose that we condition on $X_1,\ldots,X_n$. Since $X_i$ is likely to be very close to $X_{N(i)}$, the random variables $Y_i$ and $Y_{N(i)}$ are likely to be approximately i.i.d.~after this conditioning. This is the second key observation (which is tricky to make rigorous in the absence of any assumptions on the law of $(X,Y)$), which leads to the approximation
\begin{align*}
\ee[(1_{\{t\le Y_i\}} - 1_{\{t\le Y_{N(i)}\}})^2|X_1,\ldots,X_n] &\approx 2\var(1_{\{t\le Y_i\}}|X_1,\ldots,X_n) \\
&= 2\var(1_{\{t\le Y_i\}}|X_i). 
\end{align*}
This gives
\begin{align*}
\ee(1_{\{t\le Y_i\}} - 1_{\{t\le Y_{N(i)}\}})^2 &\approx  2\ee[\var(1_{\{t\le Y\}}|X)]\\
&= 2\var(1_{\{t\le Y\}}) - 2 \var(\ee(1_{\{t\le Y\}}|X)). 
\end{align*}
Combining this with \eqref{fform}, we get
\begin{align*}
\ee|F(Y_i)-F(Y_{N(i)})| &\approx \int 2[\var(1_{\{t\le Y\}}) -\var(\ee(1_{\{t\le Y\}}|X))]d\mu(t). 
\end{align*}
But note that $\var(1_{\{t\le Y\}}) = F(t)(1-F(t))$, and $F(Y)\sim \textup{Uniform}[0,1]$. Thus, 
\[
\int \var(1_{\{t\le Y\}}) d\mu(t) = \int F(t)(1-F(t)) d\mu(t) = \int_0^1 x(1-x)dx = \frac{1}{6}.
\] 
Therefore by \eqref{xiapprox},
\begin{align*}
\ee(\xi_n(X,Y)) \approx 6 \int \var(\ee(1_{\{t\le Y\}}|X))d\mu(t) = \xi(X,Y), 
\end{align*}
where the last identity holds because $\int \var(1_{\{t\le Y\}}) d\mu(t)  = 1/6$, as shown above. This establishes the convergence of $\ee(\xi_n(X,Y))$ to $\xi(X,Y)$. Concentration inequalities are then used to show that $\xi_n(X,Y)-\ee(\xi_n(X,Y))\to 0$ almost surely.

\section{Proof of Theorem \ref{mainthm}}\label{proof1}
Throughout this proof and the rest of the manuscript, we will abbreviate $\xi_n(X,Y)$ as $\xi_n$ and $\xi(X,Y)$ as $\xi$. 
For  $t\in \rr$, let $F(t) := \pp(Y\le t)$ and $G(t) := \pp(Y\ge t)$. Let $\mu$ be the law of $Y$. By the existence of regular conditional probabilities on regular Borel spaces (see for example \cite[Theorem 2.1.15 and Exercise 5.1.16]{durrett10}), for each Borel set $A\subseteq \rr$ there is a measurable map $x \mapsto \mu_x(A)$ from $\rr$ into $[0,1]$, such that 
\begin{enumerate}
\item for any $A$, $\mu_X(A)$ is a version of $\pp(Y\in A|X)$, and 
\item with probability one, $\mu_X$ is a probability measure on $\rr$. 
\end{enumerate}
In the above sense, $\mu_x$ is the conditional law of $Y$ given $X=x$. For each $t$, let $G_x(t) := \mu_x([t, \infty))$, and define
\begin{align}\label{qdef}
Q := \int \var(G_X(t)) d\mu(t). 
\end{align}
(Since $t\mapsto \ee(G_X(t))$ and $t\mapsto \ee(G_X(t)^2)$ are both non-increasing maps, they are measurable. Therefore $t\mapsto \var(G_X(t))$ is also measurable, and so the above integral is well-defined.)
\begin{lmm}\label{indepthm1}
Let $Q$ be as above. Then $Q =0$ if and only if $X$ and $Y$ are independent. 
\end{lmm}
\begin{proof}
If $X$ and $Y$ are independent, then for any $t$, $\pp(Y\ge t|X)=\pp(Y\ge t)$ almost surely. Thus, $G_X(t) = G(t)$ almost surely, and so $\var(G_X(t)) =0$. Consequently, $Q=0$. 

Conversely, suppose that $Q=0$. Then there is a Borel set $A\subseteq \rr$ such that $\mu(A)=1$ and $\var(G_X(t))=0$ for every $t\in A$. Since $\ee(G_X(t))=G(t)$, $G_X(t)=G(t)$ almost surely for each $t\in A$.  We claim that $A$ can be chosen to be the whole of $\rr$. 

To show this, take any $t\in \rr$. If $\mu(\{t\})>0$, then clearly $t$ must be a member of $A$ and there is nothing more to prove. So assume that $\mu(\{t\})=0$. This implies that $G$ is right-continuous at $t$. 

There are two possibilities. First, suppose that $G(s)<G(t)$ for all $s>t$. Then for each $s>t$, $\mu([t,s)) >0$, and hence $A$ must intersect $[t,s)$. This shows that there is a sequence $r_n$ in $A$ such that $r_n$ decreases to $t$. Since $G_X(r_n)=G(r_n)$ almost surely for each $n$, this implies that with probability one,
\[
G_X(t) \ge \lim_{n\to \infty} G_X(r_n)= \lim_{n\to\infty}G(r_n)=G(t).
\]
But $\ee(G_X(t))=G(t)$. Thus, $G_X(t) = G(t)$ almost surely. 

The second possibility is that there is some $s>t$ such that $G(s)=G(t)$. Take the largest such $s$, which exists because $G$ is left-continuous. If $s=\infty$, then $G(t)=G(s)=0$, and hence $G_X(t)=0$ almost surely because $\ee(G_X(t))=G(t)$. Suppose that $s<\infty$. Then either $\mu(\{s\})>0$, which implies that $G_X(s)=G(s)$ almost surely, or $\mu(\{s\})=0$ and $G(r)<G(s)$ for all $r>s$, which again implies that $G_X(s)=G(s)$ almost surely, by the previous paragraph. Therefore in either case, with probability one,
\[
G_X(t)\ge G_X(s)=G(s)=G(t).
\]
Since $\ee(G_X(t))=G(t)$, this implies that $G_X(t)=G(t)$ almost surely.

This completes the proof of our claim that for each $t\in \rr$, $G_X(t)=G(t)$ almost surely. Therefore, for any $t\in \rr$ and any Borel set $B\subseteq \rr$,
\begin{align*}
\pp(\{Y\ge t\}\cap \{X\in B\}) &= \ee(\pp(Y\ge t|X) 1_{\{X\in B\}})\\
&= G(t) \pp(X\in B) = \pp(Y\ge t) \pp(X\in B).
\end{align*} 
This proves that $Y$ and $X$ are independent. 
\end{proof}
\begin{cor}\label{ycor}
If $Y$ is not a constant, then $\int G(t)(1-G(t))d\mu(t)>0$. 
\end{cor}
\begin{proof}
In Lemma \ref{indepthm1}, take $X=Y$. Then $G_X(t) = 1_{\{X\ge t\}}$, and hence $\var(G_X(t)) = G(t)(1-G(t))$. But if $Y$ is not a constant, then $Y$ is not independent of itself. Hence Lemma \ref{indepthm1} implies that $Q>0$, which gives what we want.
\end{proof}

Let $X_1,X_2,\ldots$ be an infinite sequence of i.i.d.~copies of $X$. For each $n\ge 2$ and each $1\le i\le n$, let $X_{n,i}$ be the element of the set  $\{X_j: 1\le j\le n, j\ne i\}$ that is immediately to the right of $X_i$. If there is no such element, then let $X_{n,i}=X_i$.
\begin{lmm}\label{nnlmm}
With probability one, $X_{n,1}\to X_1$ as $n\to \infty$.
\end{lmm}
\begin{proof}
 Let $\nu$ be the law of $X$. Let $A$ be the set of all $x\in \rr$ such that $\nu([x, y)) >0$ for any $y>x$. First, we show that $\nu(A^c)=0$. Let $K$ be the support of $\nu$ and let $B:=A^c\cap K$. Since $\nu(K^c)=0$, it suffices to show that $\nu(B)=0$. 
 
 Take any $x\in B$. Since $x\in A^c$, there is some $y>x$ such that $\nu([x,y))=0$. For each $x\in B$, choose such a point $y_x$. We claim that the intervals $[x,y_x)$, as $x$ ranges over $B$, are disjoint. To see this, take any distinct $x,x'\in B$, $x<x'$. If $[x,y_x)$ and $[x', y_{x'})$ are not disjoint, then $x'\in (x,y_x)$. But $\nu((x,y_x)) \le \nu([x, y_x)) = 0$. This contradicts the fact that $x'\in K$. Thus, we have established that the intervals $[x,y_x)$ are disjoint. But this implies that there can be at most countably many such intervals. Thus, $B$ is at most countable. But for any $x\in B$, $\nu(\{x\})\le \nu([x,y_x))=0$. This proves that $\nu(B)=0$, and hence $\nu(A^c)=0$.
 
Take any $\ve>0$. Let $I$ be the interval $[X_1, X_1+\ve)$. Then 
\begin{align*}
\pp(|X_1-X_{n,1}|\ge \ve|X_1) &\le (1-\nu(I))^{n-1}.
\end{align*}
Since $X_1\in A$ almost surely, it follows that $\nu(I)>0$ almost surely. Thus,
\begin{align*}
\lim_{n\to\infty} \pp(|X_1-X_{n,1}|\ge \ve|X_1)=0
\end{align*}
almost surely, and hence
\[
\lim_{n\to\infty} \pp(|X_1-X_{n,1}|\ge \ve)=0.
\] 
This proves that $|X_1-X_{n,1}|\to 0$ in probability. But $|X_1-X_{n,1}|$ is decreasing in $n$ after the first time some $X_j$ is drawn that is $\ge X_1$ (and there will always be such a time, since $\nu(I)>0$). Therefore $|X_1-X_{n,1}|\to 0$ almost surely.
\end{proof}

\begin{lmm}\label{cplmm}
For any measurable function $f:\rr \to [0,\infty)$,
\[
\ee(f(X_{n,1}))\le 2\ee(f(X_1)).
\]
\end{lmm}
\begin{proof}
Consider a particular realization of $X_1,\ldots,X_n$. In this realization, take any $i$ and $j$ such that $X_{n,i}=X_j$ and $X_j\ne X_i$. We claim that for any $j$, there can be at most one such $i$. Take any $k\notin\{i,j\}$. Then $X_k$ cannot lie in the interval $[X_i, X_j)$, because that would contradict the fact that $X_{n,i}=X_j$. If $X_k< X_i$, then $X_{n,k}\ne X_j$ because $X_i$ is closer to $X_k$ on the right than $X_j$. On the other hand, if $X_k> X_j$, then obviously $X_{n,k}\ne X_j$. Thus, we conclude that for any $j$, there can be at most one $i$ such that $X_{n,i}=X_j$ and~$X_i\ne X_j$. 

Now observe that  since $f$ is nonnegative, 
\begin{align*}
\ee(f(X_{n,i})) &\le \ee(f(X_i)) + \ee(f(X_{n,i})1_{\{X_{n,i}\ne X_i\}})\\
&\le \ee(f(X_i)) + \sum_{j=1}^n \ee(f(X_j)1_{\{X_j = X_{n,i},\, X_j\ne X_i\}}).
\end{align*}
Combining the two observations and using symmetry, we get
\begin{align*}
\ee(f(X_{n,1})) &= \frac{1}{n}\sum_{i=1}^n \ee(f(X_{n,i}))\\
&\le \frac{1}{n}\sum_{i=1}^n \ee(f(X_i)) + \frac{1}{n}\sum_{i=1}^n\sum_{j=1}^n \ee(f(X_j)1_{\{X_j = X_{n,i},\, X_j\ne X_i\}})\\
&= \ee(f(X_1)) + \frac{1}{n}\sum_{j=1}^n\ee\biggl(f(X_j)\sum_{i=1}^n 1_{\{X_j = X_{n,i},\, X_j\ne X_i\}}\biggr)\\
&\le \ee(f(X_1)) + \frac{1}{n}\sum_{j=1}^n\ee(f(X_j)) = 2\ee(f(X_1)),
\end{align*}
which completes the proof of the lemma. 
\end{proof}

For the next result, we will need the following version of Lusin's theorem (proved, for example, by combining \cite[Theorem 2.18 and Theorem 2.24]{rudin87}). 
\begin{lmm}[Special case of Lusin's theorem]\label{lusinthm}
Let $f:\rr\to \rr$ be a measurable function and $\nu$ be a probability measure on $\rr$. Then, given any $\ve>0$, there is a compactly supported continuous function $g:\rr\to \rr$ such that $\nu(\{x: f(x)\ne g(x)\}) <\ve$.
\end{lmm}
\begin{lmm}\label{nnthm}
For any measurable $f:\rr\to\rr$, $f(X_1)-f(X_{n,1})$ tends to $0$ in probability as $n\to\infty$. 
\end{lmm}
\begin{proof}
Fix some $\ve>0$. Let $g$ be a function as in Lemma~\ref{lusinthm}, for the given $f$ and $\ve$, and $\nu=$ the law of $X_1$. Then note that for any $\delta >0$, 
\begin{align*}
&\pp(|f(X_1)-f(X_{n,1})|>\delta) \\
&\le \pp(|g(X_1)-g(X_{n,1})|>\delta)+ \pp(f(X_1)\ne g(X_1)) \\
&\qquad + \pp(f(X_{n,1})\ne g(X_{n,1})).
\end{align*}
By Lemma \ref{nnlmm} and the continuity of $g$, 
\[
\lim_{n\to \infty} \pp(|g(X_1)-g(X_{n,1})|>\delta) =0.
\]
By the construction of $g$,
\[
\pp(f(X_1)\ne g(X_1))< \ve.
\]
Finally, by Lemma \ref{cplmm}, 
\[
\pp(f(X_{n,1})\ne g(X_{n,1})) \le 2\pp(f(X_1)\ne g(X_1))\le 2\ve.
\]
Putting it all together, we get
\[
\limsup_{n\to\infty} \pp(|f(X_1)-f(X_{n,1})|>\delta)\le 3\ve.
\]
Since $\ve$ and $\delta$ are arbitrary, this completes the proof of the lemma.
\end{proof}

Let $\pi(i)$ be the rank of $X_i$, breaking ties at random so that $\pi$ is a permutation of $\{1,\ldots, n\}$. Define
\begin{align*}
N(i) := 
\begin{cases}
\pi^{-1}(\pi(i)+1) &\text{ if } \pi(i) < n,\\
i &\text{ if } \pi(i)=n.
\end{cases}
\end{align*}
We will now show that $\pp(X_{n,1}=X_{N(1)})\to 1$ as $n\to \infty$. For that, we need to recall the following formula.
\begin{lmm}\label{binlmm}
If $Z\sim\textup{Binomial}(m,p)$, then 
\[
\ee\biggl(\frac{1}{Z+1}\biggr) = \frac{1-(1-p)^{m+1}}{(m+1)p}.
\]
\end{lmm}
\begin{proof}
Let $x:= p/(1-p)$. Then
\begin{align*}
\ee\biggl(\frac{1}{Z+1}\biggr) &= \sum_{k=0}^m \frac{1}{k+1}{m\choose k} p^k(1-p)^{m-k}\\
&= \frac{(1-p)^m}{x}\sum_{k=0}^m {m\choose k} \frac{x^{k+1}}{k+1}\\
&=  \frac{(1-p)^m}{x}\int_0^x\sum_{k=0}^m {m\choose k} y^kdy\\
&= \frac{(1-p)^m}{x}\int_0^x(1+y)^m dy = \frac{(1-p)^m}{x}\frac{(1+x)^{m+1} - 1}{m+1}.
\end{align*}
The result is obtained by substituting the value of $x$.
\end{proof}
\begin{lmm}\label{newlmm}
$\pp(X_{n,1}=X_{N(1)})\to 1$ as $n\to \infty$. 
\end{lmm}
\begin{proof}
Let $x_1,x_2,\ldots$ be the atoms of $X$, with masses $p_1, p_2,\ldots$. Fix a realization of $X_1,\ldots, X_n$. If $X_j\ne X_1$ for all $j\ne 1$, then $X_{n,1}=X_{N(1)}$. Suppose that $X_j=X_1$ for at least one $j\ne 1$.  Let $M$ be the number of such $j$. Then with probability $1/(M+1)$, $\pi(1)$ is the highest among all such $\pi(j)$. If this does not happen, then again $X_{n,1}=X_{N(1)}$. Therefore
\[
\pp(X_{n,1}\ne X_{N(1)}) \le \ee\biggl(\frac{1}{M+1}1_{\{M\ge 1\}}\biggr). 
\]
Now let us condition on $X_1$. If $X_1\notin\{x_1,x_2,\ldots\}$, then $M=0$. If $X_1=x_i$, then conditionally $M\sim\textup{Binomial}(n-1, p_i)$. Therefore by Lemma \ref{binlmm} and the above inequality, we get
\begin{align*}
\pp(X_{n,1}\ne X_{N(1)}) &\le \sum_{i=1}^\infty \frac{1-(1-p_i)^n}{np_i}p_i. 
\end{align*}
Take any $k$. Then by the inequality $(1-x)^n\ge 1-nx$ and the above inequality, 
\begin{align*}
\pp(X_{n,1}\ne X_{N(1)})  &\le \frac{k}{n}+\sum_{i=k+1}^\infty p_i.
\end{align*}
Fixing $k$, and sending $n\to\infty$, we get
\[
\limsup_{n\to\infty}\pp(X_{n,1}\ne X_{N(1)}) \le \sum_{i=k+1}^\infty p_i.
\]
The proof is completed by sending $k\to\infty$. 
\end{proof}
\begin{cor}\label{nncor}
For any measurable $f:\rr\to\rr$, $f(X_1)-f(X_{N(1)})\to 0$ in probability as $n\to\infty$. 
\end{cor}
\begin{proof}
By Lemma \ref{nnthm}, $f(X_1)-f(X_{n,1})\to 0$ in probability. By Lemma \ref{newlmm}, $f(X_{n,1})-f(X_{N(1)})\to 0$ in probability. The claim is proved by adding the two. 
\end{proof}
For each $t\in \rr$, let
\[
F_n(t) := \frac{1}{n}\sum_{i=1}^n1_{\{Y_i\le t\}}, \ \ \ 
G_n(t) := \frac{1}{n}\sum_{i=1}^n1_{\{Y_i\ge t\}}.
\]
Define 
\[
Q_n := \frac{1}{n}\sum_{i=1}^n \min\{F_n(Y_i), F_n(Y_{N(i)})\} - \frac{1}{n}\sum_{i=1}^nG_n(Y_i)^2. 
\]
\begin{lmm}\label{expthm}
Let $Q_n$ be defined as above, and $Q$ be the quantity defined in equation \eqref{qdef}. Then $\lim_{n\to\infty} \ee(Q_n) = Q$. 
\end{lmm}
\begin{proof}
Let  
\[
Q_n' := \frac{1}{n}\sum_{i=1}^n \min\{F(Y_i), F(Y_{N(i)})\} - \frac{1}{n}\sum_{i=1}^nG(Y_i)^2. 
\]
and let 
\[
\Delta_n := \sup_{t\in \rr} |F_n(t)-F(t)| + \sup_{t\in \rr} |G_n(t)-G(t)|.
\]
Then by the triangle inequality, 
\[
|Q_n' - Q_n|\le 3\Delta_n.
\]
On the other hand, by the Glivenko--Cantelli theorem, $\Delta_n \to 0$ almost surely as $n\to\infty$. Since $\Delta_n$ is bounded by $2$, this implies that 
\[
\lim_{n\to \infty} \ee|Q_n'-Q_n|=0.
\]
Thus, it suffices to show that $\ee(Q_n')$ converges to $Q$. First, notice that
\begin{align*}
\min\{F(Y_1),F(Y_{N(1)})\} &= \int 1_{\{ Y_1\ge t\}} 1_{\{Y_{N(1)}\ge t\}}d\mu(t).
\end{align*}
Let $\mf$ be the $\sigma$-algebra generated by the $X_i$'s and the randomness used for breaking ties in the selection of $\pi$. Then for any $t$,
\begin{align*}
\ee(1_{\{Y_1\ge t\}} 1_{\{Y_{N(1)}\ge t\}}|\mf) &= G_{X_1}(t) G_{X_{N(1)}}(t). 
\end{align*}
Now recall that by the properties of the regular conditional probability $\mu_x$, the map $x \mapsto G_x(t)$ is measurable. Therefore by the above identity and Corollary~\ref{nncor}, and the boundedness of $G_x$, we have 
\begin{align*}
\lim_{n\to\infty}\ee (1_{\{Y_1\ge t\}} 1_{\{Y_{N(1)}\ge t\}}) &= \lim_{n\to\infty}\ee (G_{X_1}(t) G_{X_{N(1)}}(t))\\
&= \ee(G_X(t)^2).
\end{align*}
Thus,
\begin{align*}
\lim_{n\to \infty} \ee(Q_n')&=\int_{\rr} (\ee(G_X(t)^2) - G(t)^2) d\mu(t).
\end{align*}
Since $\ee(G_X(t))=G(t)$, this completes the proof of the lemma. 
\end{proof}
\begin{lmm}\label{concthm}
There is a positive universal constant $C$ such that for any $n$ and any $t\ge 0$,
\[
\pp(|Q_n - \ee(Q_n)|\ge t) \le 2e^{-Cnt^2}.
\]
\end{lmm}
\begin{proof}
Throughout this proof, $C$ will denote any universal constant. The value of $C$ may change from line to line. First, we will prove the claim under the assumption that $X$ has a continuous distribution, so that no randomization is involved in the definitions of $\pi$ and the $N(i)$'s. 

Assume continuity, and suppose that for some $i\le n$, $(X_i,Y_i)$ is replaced by a different value $(X_i', Y_i')$. Then there are at most three indices $j$ such that the value of $N(j)$ changes after the replacement, and exactly one index, $j=i$, where $Y_j$ changes. Moreover, there can be at most one index $j$ such that $N(j)=i$, both before and after the replacement. Lastly, for each $t$, $G_n(t)$ and $F_n(t)$ change by at most $1/n$. This shows that $Q_n$  changes by at most $C/n$ due to this replacement. The result now follows easily by the bounded difference concentration inequality~\cite{mcdiarmid89}.

Now consider the general case. Let $Z_1,\ldots, Z_n$ be i.i.d.~Uniform$[0,1]$ random variables. For each $\ve>0$, define 
\[
X_i^{\ve} := X_i + \ve Z_i.
\]
Define $Q_n^\ve$ using $(X_1^\ve,Y_1),\ldots, (X_n^\ve, Y_n)$, by the same formula that was used for defining $Q_n$ using $(X_1, Y_1),\ldots, (X_n, Y_n)$. Then by the first part we know that
\begin{align}\label{qnetail}
\pp(|Q_n^\ve - \ee(Q_n^\ve)|\ge t) \le 2e^{-Cnt^2},
\end{align}
where the important thing is that $C$ has no dependence on $\ve$. Now construct a random permutation $\pi$ as follows. Given a realization of $X_1,\ldots,X_n$, let
\[
\ve^* := \frac{1}{2}\min\{|X_i-X_j|: 1\le i,j\le n, X_i\ne X_j\}.
\] 
Having produced $\ve^*$ as above, define $\pi$ to be the rank vector of $X_1^{\ve^*},\ldots,X_n^{\ve^*}$. Notice that if $X_i<X_j$ for some $i$ and $j$, then it is guaranteed that $X_i^{\ve^*} < X_j^{\ve^*}$. From this, it is not hard to see that $\pi$ is a rank vector for $X_1,\ldots,X_n$ where ties are broken uniformly at random. On the other hand, the construction also guarantees that $\pi$ is the rank vector $X_1^\ve,\ldots, X_n^\ve$ for all $\ve \le \ve^*$. Thus, if $Q_n$ is defined using this $\pi$, then $Q_n^\ve = Q_n$ for all $\ve\le \ve^*$. Consequently, $Q_n^\ve \to Q_n$ almost surely as $\ve \to 0$. Using the uniform boundedness of $Q_n^\ve$, it is now easy to deduce the tail bound for $Q_n$ from the inequality~\eqref{qnetail}.
\end{proof}
Combining Lemmas \ref{expthm} and \ref{concthm}, we get the following corollary.
\begin{cor}\label{convascor}
As $n\to \infty$, $Q_n\to Q$ almost surely. 
\end{cor}
We are now ready to prove Theorem \ref{mainthm}.
\begin{proof}[Proof of Theorem \ref{mainthm}]
Define 
\begin{equation}\label{sndefi}
S_n := \frac{1}{n}\sum_{i=1}^n G_n(Y_i)(1-G_n(Y_i)), \ \ \ 
S_n' := \frac{1}{n}\sum_{i=1}^n G(Y_i)(1 - G(Y_i)),
\end{equation}
and $\Delta_n := \sup_{t\in \rr} |G_n(t)-G(t)|$.
Then by the triangle inequality, $|S_n - S_n'| \leq 2\Delta_n$, and by the Glivenko--Cantelli theorem,  $\Delta_n\rightarrow 0$ almost surely. But by  the strong law of large numbers,  $S_n' \to \int G(t)(1 - G(t))d\mu(t)$  almost surely as $n\to\infty$, and therefore the same holds for $S_n$. By Corollary \ref{ycor}, this limit is nonzero. Therefore by this and Corollary \ref{convascor}, we get that with probability one,
\[
\lim_{n\to\infty} \frac{Q_n}{S_n} = \xi,
\]
where $\xi$ is the quantity defined in \eqref{xidef}. Now notice that if $\pi$ is the permutation used for rearranging the data in the definition of $\xi_n$, then $nF_n(Y_i) = r_{\pi(i)}$ for all $i$, and $nF_n(Y_{N(i)}) = r_{\pi(i)+1}$ for $i\ne \pi^{-1}(n)$. If $i=\pi(n)$, then $nF_n(Y_i)=nF_n(Y_{N(i)})=r_n$. Therefore
\begin{align*}
\frac{1}{n}\sum_{i=1}^n \min\{F_n(Y_i), F_n(Y_{N(i)})\} &= \frac{1}{n^2}\sum_{i\ne \pi^{-1}(n)} \min\{r_{\pi(i)}, r_{\pi(i)+1}\} + \frac{r_n}{n^2}.
\end{align*}
By the identity $\min\{a,b\} = \frac{1}{2}(a+b - |a-b|)$, this gives 
\begin{align*}
&\frac{1}{n}\sum_{i=1}^n \min\{F_n(Y_i), F_n(Y_{N(i)})\} \\
&= \frac{1}{2n^2}\sum_{i\ne \pi^{-1}(n)} (r_{\pi(i)}+r_{\pi(i)+1} - |r_{\pi(i)}-r_{\pi(i)+1}|) + \frac{r_n}{n^2}\\
&= \frac{1}{n^2} \sum_{i=1}^n r_i -\frac{1}{2n^2}\sum_{i=1}^{n-1}|r_{i+1}-r_i| + \frac{r_n-r_1}{2n^2}. 
\end{align*}
On the other hand,
\begin{align*}
S_n &= \frac{1}{n^3}\sum_{i=1}^n l_i(n-l_i), \ \ \ \frac{1}{n}\sum_{i=1}^n G_n(Y_i)^2 = \frac{1}{n^3}\sum_{i=1}^n l_i^2,  
\end{align*}
and
\begin{align}\label{rlid}
\sum_{i=1}^n r_i &= \sum_{i=1}^n \sum_{j=1}^n 1_{\{Y_{(j)}\le Y_{(i)}\}} = \sum_{j=1}^n \sum_{i=1}^n 1_{\{Y_{(j)}\le Y_{(i)}\}} = \sum_{j=1}^n l_j.
\end{align}
Combining the above observations, we get
\begin{align*}
\frac{Q_n}{S_n} &= \xi_n + \frac{r_n-r_1}{2n^2S_n}.
\end{align*}
In particular, 
\begin{align*}
\biggl|\frac{Q_n}{S_n} - \xi_n\biggr| &\le \frac{1}{2nS_n}.  
\end{align*}
Since $S_n$ converges to a nonzero limit, this proves that $\xi_n \to \xi$ almost surely. Since for each $t$,   
\[
G(t)(1-G(t)) = \var(1_{\{Y\geq t\}})\ge \var(\pp(Y\geq t| X)),
\]
we conclude that $0\le \xi\le 1$.

Lemma \ref{indepthm1} shows that $\xi=0$ if and only if $X$ and $Y$ are independent. On the other hand, if $Y$ is a function of $X$, say $Y = f(X)$ almost surely, then 
\begin{align*}
\int\var (\pp(Y \geq t| X))d\mu(t) &= \int \var (1_{\{f(X)\geq t\}})d\mu(t)\\
&= \int_{\rr} \pp(f(X)\geq t)(1 - \pp(f(X)\geq t)) d\mu(t) \\
&= \int G(t)(1-G(t))d\mu(t),
\end{align*}
which shows that $\xi=1$. Conversely, suppose that $\xi=1$. Then by the law of total variance,
\begin{align*}
0 &= 1-\xi = \int  [\var(1_{\{Y\ge t\}}) - \var(\pp(Y\ge t|X))] d\mu(t)\\
 &= \int \ee(\var(1_{\{Y\ge t\}}|X))d\mu(t)\\
 &= \int \ee(G_X(t)(1-G_X(t))) d\mu(t).
\end{align*}
This implies that $\pp(E)=1$, where $E$ is the event
\begin{equation}\label{edef}
\int G_X(t)(1 - G_X(t)) d\mu(t) = 0.
\end{equation}
Let $A$ be the support of $\mu$. Define
\[
a_x := \sup\{t: G_x(t) = 1\},  \ \ \ b_x := \inf\{t: G_x(t)=0\},
\]
so that $a_x\le b_x$. By the measurability of $x\mapsto G_x(t)$ and the fact that $a_x\ge t$ if and only if $G_x(t)=1$, it follows  that $x \mapsto a_x$ is a measurable map. Similarly, $x\mapsto b_x$ is also measurable.

Now suppose that the event $\{a_X<b_X\}\cap E$ takes place. Since $G_X(t)\in (0,1)$ for all $t\in (a_X, b_X)$, the condition \eqref{edef} implies that $\mu((a_X, b_X)) = 0$. Since $(a_X, b_X)$ is an open interval, this implies that $(a_X, b_X) \subseteq A^c$.  On the other hand, under the given circumstance, we also have $\pp(Y\in (a_X, b_X)|X) > 0$. Thus $\pp(Y\in A^c|X)>0$. 

The above argument shows that if $\pp(\{a_X< b_X\}\cap E)>0$, then $\pp(Y\in A^c)>0$. But this is impossible, since $A$ is the support of $\mu$. Therefore  $\pp(\{a_X< b_X\}\cap E)=0$. But $\pp(E)=1$. Therefore $\pp(a_X = b_X)=1$. Thus, $Y=a_X$ almost surely.  This completes the proof of Theorem \ref{mainthm}. 
\end{proof}

\section{Preparation for the proof of Theorem \ref{cltthm}}\label{prepsec}
In this section we  prove some preparatory lemmas for the proof of Theorem \ref{cltthm}. Recall the numbers $R(i)$ and $L(i)$ defined in equation \eqref{rldef}. Let $\pi$ be a rank vector for the $X_i$'s, chosen uniformly at random from all available choices if there are ties. First, note that since $X$ and $Y$ are independent, $\pi^{-1}$ is a uniform random permutation that is independent of $Y_1,\ldots,Y_n$. Let $\tau:=\pi^{-1}$, and let 
\[
D_n := \sum_{i=1}^{n-1} a_i,
\]
where 
\[
a_i := \min\{R(\tau(i)), R(\tau(i+1))\}. 
\]
Also, for convenience, let
\[
b_{i,j} := \min\{R(i), R(j)\}.
\]
In the following, $O(n^{-\alpha})$ will denote any quantity whose absolute value is bounded above by $Cn^{-\alpha}$ for some universal constant $C$. 
Let $\ee'$, $\var'$ and $\cov'$ denote conditional expectation, conditional variance and conditional covariance given $Y_1,\ldots,Y_n$. 
\begin{lmm}\label{dmean}
\begin{align*}
\ee'(D_n) &= \frac{1}{n}\sum_{i=1}^n L(i)(L(i)-1).
\end{align*}
\end{lmm}
\begin{proof}
Take any $1\le i\le n-1$. Since $(\tau(i), \tau(i+1))$ is uniformly distributed over all pairs $(j,k)$ where $j$ and $k$ are distinct, we have
\begin{align}\label{aiexp}
\ee'(a_i) &= \frac{1}{n(n-1)} \sum_{1\le j\ne k\le n} b_{j,k}
\end{align}
Since $R(i) = \sum_{j=1}^n 1_{\{Y_j\le Y_i\}}$, this gives
\begin{align}
\ee'(a_i) &= \frac{1}{n(n-1)} \sum_{1\le j\ne k\le n}\sum_{l=1}^n 1_{\{Y_l\le Y_j, \, Y_l\le Y_k\}}\notag\\
&= \frac{1}{n(n-1)} \biggl(\sum_{1\le j, k\le n}\sum_{l=1}^n 1_{\{Y_l\le Y_j, \, Y_l\le Y_k\}} - \sum_{j=1}^n\sum_{l=1}^n 1_{\{Y_l\le Y_j\}}\biggr)\notag\\
&= \frac{1}{n(n-1)} \biggl(\sum_{l=1}^n \sum_{1\le j, k\le n}1_{\{Y_l\le Y_j, \, Y_l\le Y_k\}} - \sum_{l=1}^n\sum_{j=1}^n 1_{\{Y_l\le Y_j\}}\biggr)\notag \\
&= \frac{1}{n(n-1)} \biggl(\sum_{l=1}^n L(l)^2 - \sum_{l=1}^nL(l)\biggr).\notag
\end{align}
The proof is now  completed by adding over $i$. 
\end{proof}
\begin{lmm}\label{vardn}
$\var'(D_n)=V_n + O(n^2)$, where 
\begin{align*}
V_n &:= \frac{1}{n}\sum_{p,q=1}^nb_{p,q}^2 - \frac{2}{n^2}\sum_{p,q,r=1}^nb_{p,q} b_{p,r}  +\frac{1}{n^3} \sum_{p,q,r,s=1}^nb_{p,q} b_{r,s}.
\end{align*}
\end{lmm}
\begin{proof}
Take any $1\le i<j\le n-1$. First, suppose that $i+1<j$. Then $(\tau(i), \tau(i+1), \tau(j), \tau(j+1))$ is uniformly distributed over all quadruples of distinct $(p,q,r,s)$. Thus, 
\begin{align*}
\ee'(a_ia_j) &= \frac{1}{(n)_4} \sideset{}{'}\sum_{p,q,r,s}b_{p,q} b_{r,s},
\end{align*}
where $(n)_4:=n(n-1)(n-2)(n-3)$, and $\sum'$ denotes sum over distinct $p,q,r,s$. Therefore by \eqref{aiexp}, 
\begin{align*}
&\cov'(a_i, a_j) = \frac{1}{(n)_4} \sideset{}{'}\sum_{p,q,r,s}b_{p,q} b_{r,s} - \biggl(\frac{1}{(n)_2} \sideset{}{'}\sum_{p,q}b_{p,q} \biggr)^2\\
&= \biggl(\frac{1}{(n)_4} -\frac{1}{(n)_2^2}\biggr) \sideset{}{'}\sum_{p,q,r,s}b_{p,q}b_{r,s} - \frac{1}{(n)_2^2}\biggl(\biggl( \sideset{}{'}\sum_{p,q}b_{p,q} \biggr)^2 -  \sideset{}{'}\sum_{p,q,r,s}b_{p,q}b_{r,s}\biggr)\\
&= \frac{4n}{(n)_2(n)_4} \sideset{}{'}\sum_{p,q,r,s}b_{p,q}b_{r,s} - \frac{4}{(n)_2^2}\sideset{}{'}\sum_{p,q,r}b_{p,q}b_{p,r} + O(1)\\
&= \frac{4}{n^5} \sum_{p,q,r,s}b_{p,q}b_{r,s} - \frac{4}{n^4}\sum_{p,q,r}b_{p,q}b_{p,r} + O(1).
\end{align*}
Next, suppose that $i+1=j$. Then
\begin{align*}
\cov'(a_i, a_j) &= \frac{1}{(n)_3} \sideset{}{'}\sum_{p,q,r}b_{p,q} b_{p,r} - \biggl(\frac{1}{(n)_2} \sideset{}{'}\sum_{p,q}b_{p,q} \biggr)^2\\
&= \frac{1}{n^3} \sum_{p,q,r}b_{p,q} b_{p,r} - \frac{1}{n^4} \sum_{p,q,r,s}b_{p,q} b_{r,s} + O(n).
\end{align*}
Similarly, if $i=j$, then 
\begin{align*}
\cov'(a_i, a_j) &= \frac{1}{(n)_2} \sideset{}{'}\sum_{p,q}b_{p,q}^2 - \biggl(\frac{1}{(n)_2} \sideset{}{'}\sum_{p,q}b_{p,q} \biggr)^2\\
&= \frac{1}{n^2} \sum_{p,q}b_{p,q}^2 - \frac{1}{n^4} \sum_{p,q,r,s}b_{p,q} b_{r,s} + O(n).
\end{align*}
The proof is  completed by adding up $\cov'(a_i,a_j)$ over all $1\le i,j\le n-1$.
\end{proof}
\begin{lmm}\label{vnlim}
As $n\to \infty$, $\var'(D_n)/n^3$ converges almost surely to the deterministic limit
\begin{align*}
\ee(\phi(Y_1,Y_2)^2- 2\phi(Y_1,Y_2)\phi(Y_1,Y_3) + \phi(Y_1, Y_2)\phi(Y_3,Y_4)),
\end{align*}
where $\phi(y, y') := \min\{F(y), F(y')\}$ and $Y_1, Y_2, Y_3,Y_4$ are i.i.d.~copies of $Y$.
\end{lmm}
\begin{proof}
Throughout this proof, $C$ will be used to denote any universal constant. 
Let $V_n$ be as in Lemma \ref{vardn}. It is a function of the $Y_i$'s only. Notice that if one $Y_i$ is replaced by some other value $Y_i'$, then each $R(j)$ changes by at most $1$ for $j\ne i$, and $R(i)$ changes by at most $n$. Therefore $b_{p,q}$ changes by at most $1$ if $p\ne i$ and $q\ne i$, and by at most $n$ if one or both of the indices are equal to $i$. Moreover, the $b_{pq}$'s are all bounded by $n$. Thus, changing one $Y_i$ to $Y_i'$ changes $V_n$ by at most $Cn^2$. Therefore by the bounded difference inequality,
\begin{align*}
\pp(|V_n-\ee(V_n)|\ge t)\le 2e^{-Ct^2/n^5}
\end{align*}
for every $t$. Consequently, $(V_n-\ee(V_n))/n^3\to 0$ almost surely as $n\to\infty$.

On the other hand, note that $b_{p,q}/n = \min\{F_n(Y_p), F_n(Y_q)\}$, where $F_n$ is the empirical distribution function of the $Y_i$'s. By the Glivenko--Cantelli theorem, $F_n\to F$ uniformly with probability one, where $F$ is the cumulative distribution function of $Y$. From this, it is easy to see that $\ee(V_n)/n^3$ converges to the displayed limit.
\end{proof}
\begin{lmm}\label{vnpos}
If $Y$ is not a constant, the limit in Lemma \ref{vnlim} is strictly positive.
\end{lmm}
\begin{proof}
 Let us denote the limit by $v$. 
Let $Y'$ be an independent copy of $Y$, and define
\[
\psi(y) := \ee(\phi(y,Y')) = \ee(\phi(Y,Y')|Y=y). 
\]
Also, let $m := \ee(\phi(Y,Y')) = \ee(\psi(Y))$. Then $v$ can be expressed as
\begin{equation}\label{vform1}
v = \ee(\phi(Y, Y')^2) - 2\ee(\psi(Y)^2) + m^2. 
\end{equation}
Now, 
\begin{align*}
&\ee(\phi(Y,Y') - \psi(Y)-\psi(Y')+ m)^2\\
&= \ee(\phi(Y,Y')^2 + \psi(Y)^2 + \psi(Y')^2 + m^2 - 2\phi(Y,Y')\psi(Y) \\
&\qquad - 2\phi(Y,Y') \psi(Y') + 2\phi(Y,Y')m + 2\psi(Y)\psi(Y') \\
&\qquad - 2\psi(Y)m - 2\psi(Y')m).
\end{align*}
Note that $\ee(\phi(Y,Y')\psi(Y)) =\ee(\psi(Y)^2)$, and recall that $\ee(\phi(Y,Y')) = \ee(\psi(Y))=m$. The same identities hold if we exchange $Y$ and $Y'$. Using these facts, it is now easy to verify that the above expression is actually equal to the right side of \eqref{vform1}. Thus,
\[
v = \ee(\phi(Y,Y') - \psi(Y)-\psi(Y')+ m)^2.
\]
Hence $v\ge 0$, and $v=0$ if and only if $\phi(Y,Y') = \psi(Y)+\psi(Y')-m$ almost surely. Suppose that this is true. Then almost surely for each $i\ge 2$, 
\begin{equation}\label{phieq}
\phi(Y_1, Y_i) = \psi(Y_1) + \psi(Y_i) - m,
\end{equation}
where $Y_1,Y_2,\ldots$ are i.i.d.~copies of $Y$. 
Taking the minimum over $2\le i\le n$ on both sides, we get
\begin{align*}
\min\{F(Y_1),\ldots, F(Y_n)\} = \psi(Y_1) + \min\{\psi(Y_2),\ldots,\psi(Y_n)\} - m. 
\end{align*}
Now, the minimum of a sequence of i.i.d.~bounded random variables converges almost surely to the infimum of the support. Also, $F$ and $\psi$ are bounded functions. Therefore taking $n\to\infty$ on both sides of the above, it follows that $\psi(Y_1)$ equals a constant almost surely. Therefore $\psi(Y_2)$ equals the same constant almost surely, and hence by \eqref{phieq}, $\phi(Y_1,Y_2)$ is also equal to a constant almost surely. Now, if $L(t) := \pp(F(Y)\ge t)$, then $\pp(\phi(Y_1, Y_2)\ge t) = L(t)^2$. Since $\phi(Y_1,Y_2)$ is a constant, this shows that $L(t)^2$ is $0$ or $1$ for every $t$, and hence $L(t)$ is also $0$ or $1$ for every $t$. Consequently, $F(Y)$ is a constant almost surely. 

We claim that $1$ is in the support of $F(Y)$ and hence $F(Y)=1$ almost surely. To see this, take any $\ve\in (0,1)$. We will show that $\pp(F(Y)> 1-\ve) >0$. Let $x:= \inf\{y: F(y)\ge 1- \ve/2\}$. Then $x$ is a finite real number since $F$ tends to $1$ at $\infty$ and to $0$ at $-\infty$. By the right-continuity of $F$, $F(x)\ge 1-\ve/2$. If $F$ is discontinuous at $x$, this immediately shows that $\pp(F(Y)>1-\ve) \ge \pp(Y=x)>0$. If $F$ is continuous at $x$, there is some $y<x$ such that $F(y)>1-\ve$. By the definition of $x$, $F(y)< F(x)$. Thus, $\pp(F(Y)>1-\ve)\ge \pp(Y\in (y,x)) >0$. This shows that $1$ is in the support of $F(Y)$, and hence $F(Y)=1$ almost surely. 

Since $Y$ is not a constant, there are at least two points in its support. Therefore there exist two disjoint nonempty open intervals $I$ and $J$ such that $\pp(Y\in I)$ and $\pp(Y\in J)$ are both positive. Suppose that $I$ is to the left of $J$. Then for any $y\in I$, $F(y)\le 1-\pp(Y\in J) <1$, and hence $\pp(F(Y)<1)\ge \pp(Y\in I)>0$, which contradicts the conclusion of the previous paragraph. This shows that $v>0$.
\end{proof}

\section{Proof of Theorem \ref{cltthm}}
We will continue with the notations from Section \ref{prepsec}. Let $\sigma^2$ denote the limit of $\var'(D_n)/n^3$, which by Lemmas \ref{vnlim} and \ref{vnpos}, is a deterministic positive quantity (it was called $v$ in the proof of Lemma \ref{vnpos}). Define
\[
\tilde{D}_n := \frac{D_n-\ee'(D_n)}{n^{3/2}\sigma}. 
\]
Notice that $r_i = R(\tau(i))$. Therefore by Lemma \ref{dmean}, the identity \eqref{rlid}, and the identity $\min\{a,b\}=\frac{1}{2}(a+b-|a-b|)$, we get
\begin{align*}
D_n-\ee'(D_n) &= \sum_{i=1}^{n-1}\min\{r_i, r_{i+1} \} - \frac{1}{n}\sum_{i=1}^nL(i)(L(i)-1)\\
&= \frac{1}{2}\sum_{i=1}^{n-1} (r_i+r_{i+1}-|r_{i+1}-r_i|) - \frac{1}{n}\sum_{i=1}^nl_i(l_i-1)\\
&= \sum_{i=1}^n r_i  - \frac{r_1+r_n}{2} - \frac{1}{2}\sum_{i=1}^{n-1} |r_{i+1}-r_i| - \frac{1}{n}\sum_{i=1}^nl_i(l_i-1) \\
&= \frac{1}{n}\sum_{i=1}^n l_i(n-l_i) - \frac{1}{2}\sum_{i=1}^{n-1} |r_{i+1}-r_i| + O(n). 
\end{align*}
This shows that
\begin{align*}
\xi_n 
&= \frac{D_n-\ee'(D_n)}{n^2S_n}+ O\biggl(\frac{1}{nS_n}\biggr) = \frac{\sigma}{\sqrt{n}S_n} \tilde{D}_n + O\biggl(\frac{1}{nS_n}\biggr),
\end{align*}
where $S_n$ is the quantity defined in \eqref{sndefi}. In the proof of Theorem \ref{mainthm}, we showed that $S_n \to \int G(t)(1-G(t)d\mu(t)$ almost surely, and the latter quantity is positive by Corollary \ref{ycor}.  Thus, to prove the central limit theorem for $\sqrt{n}\xi_n$, it suffices to prove the central limit theorem for $\tilde{D}_n$. The formula for the  limiting variance $\tau^2$ can be read off from the limit of $S_n$ and the formula for $\sigma$. The limiting variance is strictly positive by Lemma~\ref{vnpos}. When $Y$ is continuous, $F(Y)\sim \textup{Uniform}[0,1]$. Using this fact, an easy calculation shows that $\tau^2=2/5$.

The central limit theorem for $\tilde{D}_n$ can be proved by mimicking the proof of the main theorem of the paper \cite{cbl93}. First, replace $D_n$ by 
\begin{align*}
D_n' := \sum_{i=1}^n \min\{R(\tau(i)), R(\tau(i+1))\},
\end{align*}
where $\tau(n+1):= \tau(1)$. Since $|D_n'-D_n|\le n$,   it suffices to prove that $\tilde{D}'_n\to N(0,1)$ in distribution, where
\[
\tilde{D}_n' := \frac{D_n'-\ee'(D_n')}{n^{3/2}\sigma}. 
\]
Mimicking the main idea of \cite{cbl93}, we define
\[
f(\tau(i+1)) := \ee'(\min\{R(\tau(i)), R(\tau(i+1))\}|\tau(i+1)),
\]
and observe that 
\[
\ee'(D_n') = n \ee [f(\tau(1))] = \sum_{i=1}^n f(i) = \sum_{i=1}^n f(\tau(i)).
\]
Thus, 
\[
\tilde{D}_n' = \frac{\sum_{i=1}^n \beta_i}{n^{3/2}\sigma},
\]
where $\beta_i := \min\{R(\tau(i)), R(\tau(i+1))\} - f(\tau(i))$. Since $|D_n-D_n'|\le n$,   $\var'(D_n')/n^3$ converges almost surely to $\sigma^2$. Using these observations, we can proceed exactly as in the proof of the main theorem of \cite{cbl93} to show that for every integer $k\ge 1$,
\begin{align}\label{dneq}
\ee'[(\tilde{D}_n')^k] \to  \ee(Z^k)  \text{ almost surely as $n\to\infty$,}
\end{align}
where $Z\sim N(0,1)$. On the other hand, a simple argument using the bounded difference inequality (viewing $\tau$ as the rank vector of i.i.d.~random variables from any continuous distribution) shows that for any $k$, 
\[
\sup_{n\ge 1} \ee|\tilde{D}_n'|^k <\infty. 
\]
Therefore by \eqref{dneq} and uniform integrability, we conclude that for every integer $k\ge 1$, 
\begin{align*}
\lim_{n\to\infty}\ee[(\tilde{D}_n')^k] =  \ee(Z^k).
\end{align*}
This completes the proof of Theorem \ref{cltthm}.

\section{Proof of Theorem \ref{estthm}}
The quantity $S_n$ define in \eqref{sndefi} is the same as $d_n$, and in the proof of Theorem \ref{mainthm} we showed that $S_n$ converges to the square-root of the denominator in the definition of $\tau^2$. Recall the quantity $V_n$ from Lemma \ref{vardn}. By Lemma \ref{vnlim}, we know $V_n/n^3$ converges almost surely to the numerator in the definition of $\tau^2$. We will now show that $a_n-2b_n+c_n^2$ is the same as $V_n/n^3$. 

From the definition of $V_n$, it is easy to see that the result will remain unchanged if we permute the $R(i)$'s and recompute $V_n$. So we can replace the $R(i)$'s by an increasing rearrangement $u_1,\ldots, u_n$. Redefine 
\[
b_{ij} := \min\{u_i, u_j\} = u_{\min\{i,j\}}.
\]
Then it is clear that
\begin{align*}
\sum_{i,j} b_{ij} &= \sum_{i=1}^n u_i + 2\sum_{1\le i<j\le n} b_{ij} \\
&= \sum_{i=1}^n u_i + 2\sum_{i=1}^n (n-i)u_i = \sum_{i=1}^n (2n-2i+1)u_i.
\end{align*}
Similarly,
\[
\sum_{i,j}b_{ij}^2 =  \sum_{i=1}^n (2n-2i+1)u_i^2.
\]
Finally,
\begin{align*}
\sum_{i,j,k} b_{ij}b_{ik} &= \sum_{i=1}^n \biggl(\sum_{j=1}^n b_{ij}\biggr)^2\\
&= \sum_{i=1}^n \biggl(\sum_{j=1}^i u_j + (n-i) u_i\biggr)^2= \sum_{i=1}^n (v_i + (n-i) u_i)^2.
\end{align*}
These expressions make it clear that $a_n-2b_n+c_n^2=V_n/n^3$. This completes the proof of convergence. Finally, to see that $\hat{\tau}_n^2$ can be computed in time $O(n\log n)$, simply observe that the computation involves only sorting and calculating cumulative sums, both of which can be done in time $O(n\log n)$. 
 
\section*{Acknowledgements}
I thank Mona Azadkia, Peter Bickel, Holger Dette, Mathias Drton, Lihua Lei, Bodhisattva Sen, Rik Sen and Steve Stigler for a number of useful comments and references. I am especially grateful to Persi Diaconis and Susan Holmes for many suggestions that greatly improved the paper. Lastly, I thank the anonymous referees for a number of suggestions that helped improve the presentation.

\end{document}